\documentclass[fullpage]{amsart}  

\usepackage{amsmath, amssymb} 
\usepackage{latexsym}
\usepackage[english]{babel}
\usepackage{hyperref}
\usepackage{amssymb}
\usepackage{amsmath}
\usepackage{amsthm}
\usepackage{a4wide}
\usepackage{amsthm,amssymb,amsmath,a4wide,graphicx,tikz}
\usepackage{subfigure}
\usepackage{caption}
\usepackage{hhline}
\usetikzlibrary{patterns}
\usepackage{float}

\usepackage{titlesec}
\titlelabel{\thetitle.\quad}

\usepackage[a4paper]{geometry}
\DeclareMathOperator{\KI}{KI}

\DeclareMathOperator{\hKI}{K\hat I}

\DeclareMathOperator{\KA}{KA}
\DeclareMathOperator{\KB}{KB}

\DeclareMathOperator{\supp}{supp}

\tikzset{
  treenode/.style = {shape=rectangle, rounded corners,
                     draw, align=center,
                     top color=white, bottom color=blue!20},
  root/.style     = {treenode, font=\Large, bottom color=red!30},
  env/.style      = {treenode, font=\ttfamily\normalsize},
  dummy/.style    = {circle,draw}
}

\makeatletter
\renewcommand*\env@matrix[1][*\c@MaxMatrixCols c]{%
  \hskip -\arraycolsep
  \let\@ifnextchar\new@ifnextchar
  \array{#1}}
\makeatother

\newtheorem{thm}{Theorem}[section]
\newtheorem{lem}[thm]{Lemma}

\newtheorem{defn}[thm]{Definition}

\newtheorem{prop}[thm]{Proposition}
\newtheorem{nrem}[thm]{Remark}

\definecolor{apricot}{rgb}{0.98, 0.81, 0.69}
\definecolor{aquamarine}{rgb}{0.5, 1.0, 0.83}
\definecolor{babyblueeyes}{rgb}{0.63, 0.79, 0.95}
\definecolor{bananamania}{rgb}{0.98, 0.91, 0.71}
\definecolor{bittersweet}{rgb}{1.0, 0.44, 0.37}
\definecolor{bluebell}{rgb}{0.64, 0.64, 0.82}
\definecolor{blush}{rgb}{0.87, 0.36, 0.51}
\definecolor{brilliantlavender}{rgb}{0.96, 0.73, 1.0}
\definecolor{darkcyan}{rgb}{0.0, 0.55, 0.55}
\definecolor{cerulean}{rgb}{0.0, 0.48, 0.65}

\begin{document}
\title{Invariant densities for random systems of the interval\\ 
}
\author[Charlene Kalle]{Charlene Kalle}
\address[Charlene Kalle]{Mathematisch Instituut, Leiden University, Niels Bohrweg 1, 2333CA Leiden, The Netherlands}
\email[Charlene Kalle]{kallecccj@math.leidenuniv.nl}
\author[Marta Maggioni]{Marta Maggioni}
\address[Marta Maggioni]{Mathematisch Instituut, Leiden University, Niels Bohrweg 1, 2333CA Leiden, The Netherlands}
\email[Marta Maggioni]{m.maggioni@math.leidenuniv.nl}


\maketitle

\begin{abstract}
For random piecewise linear systems $T$ of the interval that are expanding on average we construct explicitly the density functions of absolutely continuous $T$-invariant measures. In case the random system uses only expanding maps our procedure produces all invariant densities of the system. Examples include random tent maps, random $W$-shaped maps, random $\beta$-transformations and random L\"uroth maps with a hole.
\end{abstract}

\section{Introduction}
The Perron-Frobenius operator has been used since the seminal paper \cite{LaYo} of Lasota and Yorke to establish the existence of absolutely continuous invariant measures for deterministic dynamical systems. The same approach was also successful in the study of random dynamical systems. In the random setting, instead of a single map, a family of maps is considered from which one is selected at each iteration at random. In \cite{Pe} Pelikan gave sufficient conditions under which a random system using a finite number of piecewise $C^2$-transformations on the interval has absolutely continuous invariant measures. He also discussed the possible number of ergodic components. Around the same time a similar result was obtained by Morita in \cite{Mo}, allowing for the possibility to choose from an infinite family of maps as well. In more recent years these results have been generalised in various ways. See \cite{Bu,GoBo,BaGo1,In} for example.

\vskip .2cm

Finding an explicit formula for the density functions of these absolutely continuous invariant measures, however, is a different matter. Here the Perron-Frobenius operator can only help if one can make an educated guess. An explicit expression for the invariant density is therefore available only for specific families of maps. In 1957 R\'enyi gave in \cite{Ren1} an expression for the invariant density of the $\beta$-transformation $x \mapsto \beta x \pmod 1$ in case $\beta=\frac{1+\sqrt 5}{2}$, the golden mean. Later Parry and Gel'fond gave a general formula for the invariant density of the $\beta$-transformation in \cite{Parry60,Gel}. In \cite{DK10} generalisations of the $\beta$-transformation were considered. A more general set-up allowing different slopes was proposed in \cite{Ko} by Kopf. He introduced for any piecewise linear expanding interval map satisfying some minor restraints a matrix $M$ and associated each absolutely continuous invariant measure of the system to a vector from the null space of $M$. Some twenty years later, G\'ora developed in \cite{Go} a similar procedure for deterministic piecewise linear eventually expanding interval maps. Unless the map in question has many onto branches, the matrix involved in the procedure from \cite{Go} is of higher dimension than the one used in \cite{Ko}.

\vskip .2cm
For random maps not much is known. An exception is the random $\beta$-transformation, which was first introduced in \cite{DaKr} by Dajani and Kraaikamp and uses random combinations of two piecewise linear maps with constant slope $\beta >1$. It has a unique absolutely continuous invariant measure, as proved in \cite{DaVr2}. In \cite{Ke} Kempton gave a formula for the invariant density in case one chooses the two base maps with equal probability. Recently Suzuki extended these results in \cite{Su} to include the non-uniform Bernoulli regime as well.

\vskip .2cm
This article concerns finding explicit expressions for invariant densities of random systems. We consider any finite or countable family $\{ T_j:[0,1] \to [0,1] \}$ of piecewise linear maps that are expanding on average. The random system $T$ is given by choosing at each step one of these maps using a probability vector $\mathbf p = (p_j)$. The existence of an absolutely continuous invariant measure for such systems is guaranteed by \cite{Pe} for a finite family and by \cite{Bu, In} in the countable case. The main result of this article is that we provide a procedure to construct explicit formulas for invariant probability densities of the random system $T$. This is the content of Theorem~\ref{thmm}. The results from Theorem~\ref{thmm} cover those from \cite{Ke} and \cite{Su} regarding the expression for the invariant density as a special case. In case we assume that all maps $T_j $ are expanding, we obtain the stronger result that the procedure leading to Theorem~\ref{thmm} actually produces all absolutely continuous invariant measures of $T$. We prove this in Theorem~\ref{t:alldensities}. 

\vskip .2cm
The paper is outlined as follows. In the second section we specify our set-up and introduce the necessary assumptions and notation. The third section is devoted to the definition of a matrix $M$ and to the proof that the null space of $M$ is non-trivial. In the fourth section we prove Theorem~\ref{thmm}, relating each non-trivial vector $\gamma$ from the null space of $M$ to the density $h_\gamma$ of an absolutely continuous invariant measure of the system $T$. In the fifth section we prove Theorem~\ref{t:alldensities} on when we get all invariant densities. It is in this section that the extra difficulties that we had to overcome for dealing with random systems instead of deterministic ones, are most visible. Finally, we apply the results to some examples. In the sixth section we consider random tent maps, random $W$-shaped maps and various random $\beta$-transformations. In the last section we elaborate on a system related to representations of real numbers: a random L\"uroth map with a hole.

\section*{Acknowledgement}
We would like to thank Karma Dajani for suggesting the use of \cite{Ko}. The second author is supported by the NWO TOP-Grant No.~614.001.509.

\section{Preliminaries}\label{s:preliminaries}
Let $\Omega \subseteq \mathbb N$ and let $\{ T_j : [0,1] \rightarrow [0,1] \}_{j \in \Omega}$ be a family of piecewise linear transformations. Consider a positive probability vector $\mathbf p = (p_j)_{j \in \Omega}$, i.e., $p_j > 0$ for all $j \in \Omega$ and $\sum_{j \in \Omega} p_j=1$. We call the system $T$ a {\em random system} of the interval $[0,1]$ of probability $\mathbf p$, if for $x \in [0,1]$ and $j \in \Omega$,
$$T(x):=T_j(x) \text{ with probability } p_j.$$

\vskip .2cm
A measure $\mu_{\mathbf p}$ on $[0,1]$ is an {\em absolutely continuous invariant measure} for $T$ and $\mathbf p$ if there is a density function $h$, such that for each Borel set $A \subseteq [0,1]$ we have
\begin{equation}\label{q:invdensity}
\mu_{\mathbf p} (A) = \int_A h \, d\lambda = \sum_{j \in \Omega} p_j\mu_{\mathbf p}(T_j^{-1}A),
\end{equation}
where $\lambda$ denotes the one-dimensional Lebesgue measure. 

\vskip .2cm
Such a random system $T$ can also be described by a {\em pseudo skew-product system}. In that case, let $\sigma: \Omega^{\mathbb N} \to \Omega^{\mathbb N}$ be the left shift on sequences and define the map $R: \Omega^{\mathbb N} \times [0,1] \to \Omega^{\mathbb N} \times [0,1]$ by $R(\omega,x) = (\sigma(\omega), T_{\omega_1}x)$. If $m_{\mathbf p}$ is the $\mathbf p$-Bernoulli measure on $\Omega^{\mathbb N}$, then $m_{\mathbf p} \times \mu_{\mathbf p}$ is an invariant measure for $R$. We call $R$ the pseudo skew-product system associated to $T$.

\vskip .2cm
We put some assumptions on the systems $T$ we consider.

\vskip .2cm
\noindent {\bf (A1)} Assume that the set of all the critical points of the maps $T_j$ is {\em finite}.

\vskip .2cm
\noindent Denote these critical points by $0  = z_0 < z_1 < \cdots <z_N=1$. The points $z_i$ together specify a common partition $\{ I_i\}_{1 \leq i  \leq N}$ of subintervals of $[0,1]$, such that all maps $T_j$ are monotone on each of the intervals $I_i$. Hence, there exist $k_{i,j}, d_{i,j} \in \mathbb{R}$ such that the maps $T_{i,j}:=T_j|_{I_i}$ are given by
$$T_{i,j}(x)=k_{i,j} x +d_{i,j}.$$

\noindent {\bf (A2)} Assume that $T$ is {\em expanding on average} with respect to $\mathbf p $, i.e., assume that there is a constant $0 < \rho < 1$, such that for all $x \in [0,1]$, $\sum_{j \in \Omega} \frac{p_j}{|T'_j(x)|}  \leq \rho < 1$. This is equivalent to assuming that for each $1 \le i \le N$,
\[ \sum_{j \in \Omega} \frac{p_j}{|k_{i,j}|} \leq \rho < 1.\]

\vskip .2cm
\noindent Under these conditions the random system $T$ satisfies the conditions (a) and (b) from \cite{In}, which studies the existence of invariant densities $h$ satisfying \eqref{q:invdensity} using the Perron-Frobenius operator. For the deterministic maps $T_j:[0,1] \to [0,1]$, $j \in \Omega$, the Perron-Frobenius operator on $L^1(\lambda)$ is given by
\[ P_{T_j} f(x) = \sum_{y \in T_j^{-1}\{x\}} \frac{f(y)}{|T'_j(y)|}.\]
The {\em random Perron-Frobenius operator} is then defined by
\begin{equation}\label{q:PF}
P_T f = \sum_{j \in \Omega} p_j P_{T_j} f.
\end{equation}
The operator $P_T$ is clearly linear and positive. An $L^1(\lambda)$-function $h$ is called {\em $T$-invariant} for the random system $T$ if it is a fixed point of $P_T$, i.e., if it satisfies $P_T h =h$ $\lambda$-almost everywhere. A density function $h$ is the density of a measure $\mu_{\mathbf p}$ satisfying \eqref{q:invdensity} if and only if it is a fixed point of $P_T$. From \cite[Theorem 5.2]{In} it follows that a $T$-invariant measure $\mu_{\mathbf p}$ of the form \eqref{q:invdensity}, and hence a $T$-invariant function $h$, exists. Inoue obtained this result by showing that the operator $P_T$, applied to functions of bounded variation, satisfies a Lasota-Yorke type inequality. From the famous Ionescu-Tulcea and Marinescu Theorem one can then deduce much more than mere existence of an absolutely continuous invariant measure, it says that $P_T$ as an operator on the space of functions of bounded variation is quasi-compact. The specific implications of the quasi-compactness of $P_T$ that we use in this paper are the following. The eigenvalue $1$ of $P_T$ has a finite dimensional eigenspace. In other words, the subspace of $L^1(\lambda)$ of $T$-invariant functions is a finite-dimensional sublattice of the space of functions of bounded variation. As such, it has a finite base $H = \{ v_1, \ldots, v_r\}$ of $T$-invariant density functions of bounded variation, each corresponding to an ergodic measure, so that any other $T$-invariant $L^1(\lambda)$-function $h$ can be written as a linear combination of the $v_i$: $h = \sum_{i=1}^r c_i v_i$ for some constants $c_i \in \mathbb R$. Furthermore, if we set $U_i:= \{ x \, : \, v_i(x) >0\}$ for the \textit{support} of the function $v_i$, then each $U_i$ is forward invariant under $T$ in the sense that
\begin{equation}\label{q:forward}
\lambda \Big( U_i \triangle \bigcup_{j \in \Omega} T_j(U_i) \Big) =0,
\end{equation}
where $\triangle$ denotes the symmetric difference. Also, the sets $U_i$ are mutually disjoint and none of the sets $U_i$ can properly contain another forward invariant set. We will use these properties in the proofs from Section~\ref{s:alldensities}. An account of these implications on the operator $P_T$ can be found in \cite{Pe,Mo,Bu,In}, for example. For more information, we also refer to standard textbooks like \cite{BoGo} and \cite{LM}.

\vskip .2cm
In this article we find $T$-invariant functions $h:[0,1]\to \mathbb R$ by linking them to the vectors from the null space of a matrix $M$. To guarantee that this null space is non-trivial, we formulate three additional assumptions that are easily to verify for any given systems. Firstly, we assume that not all the lines $x \mapsto k_{i,j}x + d_{i,j}$, $1 \leq i \leq N$, with respective weights $p_j$, have a common intersection point with the diagonal. More precisely, consider for each interval $I_i$ the weighted intersection point with the diagonal
\[ x = \sum_{j \in \Omega} p_j \Big(\frac{x}{k_{i,j}} -\frac{d_{i,j}}{k_{i,j}} \Big).\]
Our third assumption states that for each $i$ there is an $n$, such that these points do not coincide.

\vskip .2cm

\noindent {\bf (A3)} Assume that for each $1 \le i \le N$, there is an $1 \le n \le N$,  such that
\[
 \frac{ \sum_{j \in \Omega} \frac{p_j}{k_{i,j}} d_{i,j}  }{1-  \sum_{j\in \Omega}  \frac{p_j}{k_{i,j}}} \neq \frac{ \sum_{j\in \Omega} \frac{p_j}{k_{n,j}}d_{n,j} }{1-  \sum_{j\in \Omega} \frac{p_j}{k_{n,j}}}.
\]

\vskip .2cm
\noindent Note that if $d_{i,j} < 0$, then $k_{i,j} > -d_{i,j}$ and if $d_{i,j}>1$, then $k_{i,j} < 1-d_{i,j}$. Hence, in all cases $|d_{i,j}| < |k_{i,j}|+1$ and by (A2),
\begin{equation}\label{q:finiteD}
\sum_{j \in \Omega} \frac{p_j}{|k_{i,j}|} |d_{i,j}| \le 1+ \rho.
\end{equation}
So, the quantities in (A3) are all finite. Our fourth assumption is on the orbits of the points 0 and 1.

\vskip .2cm
\noindent {\bf (A4)} For each $j$, assume that 
\[
d_{1,j}= \left \{
\begin{aligned}
& 0,  &\mbox{ if } k_{1,j}  >0, \cr
&1, &\mbox{ if } k_{1,j}  <0,\end{aligned}\right.
\qquad \text{ and } \qquad
d_{N,j}= \left \{
\begin{aligned}
& 1-k_{N,j}, &\mbox{ if } k_{N,j}  >0, \cr
& -k_{N,j}, &\mbox{ if } k_{N,j}  <0.\end{aligned}\right.
\]

\vskip .2cm
In other words, the points 0 and 1 are mapped to 0 or 1 under all maps $T_j$, making the system continuous at the origin, when we consider it as acting on the circle $\mathbb R / \mathbb Z$ with the points 0 and 1 identified. Since we can deal with finitely many discontinuities, (A4) is not necessary for our results to hold, but it makes computations easier. Any system not satisfying it can be extended to a system that does satisfy this condition and for which no absolutely continuous invariant measure puts weight on the added pieces. See Figure~\ref{f:extendedmap} for an illustration and see Section \ref{sec5} for a concrete example, given by the random $(\alpha,\beta)$-transformation.

\begin{figure}[h!]
\centering
\subfigure{
\begin{tikzpicture}[scale=4] 

\draw(0,0)node[below]{\small $0$}--(.2,0)node[below]{$z_1$}--(.35,0)node[below]{$z_2$}--(.5,0)node[below]{$z_3$}--(.65,0)node[below]{$z_4$}--(.75,0)node[below]{$z_5$}--(1,0)node[below]{\small $1$}--(1,1)--(.875,1)node[above]{$I_6$}--(.7,1)node[above]{$I_5$}--(.575,1)node[above]{$I_4$}--(.425,1)node[above]{$I_3$}--(.275,1)node[above]{$I_2$}--(.1,1)node[above]{$I_1$}--(0,1)node[left]{\small $1$}--(0,.7)node[left]{$a_{1,2}$}--(0,.55)node[left]{$a_{1,1}$}--(0,.4)node[left]{$b_{1,1}$}--(0,.3)node[left]{$a_{1,3}$}--(0,.15)node[left]{$b_{1,3}$}--(0,0)node[left]{$b_{1,2}$};

\draw[dotted](.2,0)--(.2,1)(.35,0)--(.35,1) (.5,0)--(.5,1)(.65,0)--(.65,1)(.75,0)--(.75,1);
\draw[dashed](.2,.7)--(0,.7)(.2,.55)--(0,.55)(.2,.4)--(0,.4)(.2,.3)--(0,.3)(.2,.15)--(0,.15);

\draw[line width=0.3mm, black] (0,0)--(.2,.3)(0,1)--(.2,.7)(.2,.15)--(.35,.8)(.2,.4)--(.35,.2)(.35,.6)--(.5,.1)(.35,.4)--(.5,.8)(.5,.4)--(.65,.2)(.65,.0)--(.75,.65)(.75,.5)--(1,1)(.75,.3)--(1,1)(0,0)--(.2,.55)(.2,0)--(.35,.6)(.35,.9)--(.5,.6)(.65,.96)--(.75,.4)(.75,.05)--(1,1);

\filldraw[draw=black, fill=black] (0,0) circle (.3pt);
\filldraw[draw=black, fill=black] (0,1) circle (.3pt);
\filldraw[draw=black, fill=black] (.2,.3) circle (.3pt);
\filldraw[draw=black, fill=black] (.2,.7) circle (.3pt);
\filldraw[draw=black, fill=black] (.35,.8) circle (.3pt);
\filldraw[draw=black, fill=black] (.35,.2) circle (.3pt);
\filldraw[draw=black, fill=black] (.5,.1) circle (.3pt);
\filldraw[draw=black, fill=black] (.5,.8) circle (.3pt);
\filldraw[draw=black, fill=black] (1,1) circle (.3pt);
\filldraw[draw=black, fill=white] (.2,.15) circle (.3pt);
\filldraw[draw=black, fill=white] (.2,.4) circle (.3pt);
\filldraw[draw=black, fill=black] (.75,.65) circle (.3pt);
\filldraw[draw=black, fill=black] (.35,.6) circle (.3pt);
\filldraw[draw=black, fill=white] (.35,.4) circle (.3pt);
\filldraw[draw=black, fill=white] (.5,.4) circle (.3pt);
\filldraw[draw=black, fill=black] (.65,.2) circle (.3pt);
\filldraw[draw=black, fill=white] (.65,.0) circle (.3pt);
\filldraw[draw=black, fill=white] (.75,.5) circle (.3pt);
\filldraw[draw=black, fill=white] (.75,.3) circle (.3pt);
\filldraw[draw=black, fill=black] (.2,.55) circle (.3pt);
\filldraw[draw=black, fill=black] (.75,.4) circle (.3pt);
\filldraw[draw=black, fill=white] (.2,0) circle (.3pt);
\filldraw[draw=black, fill=white] (.35,.9) circle (.3pt);
\filldraw[draw=black, fill=black] (.5,.6) circle (.3pt);
\filldraw[draw=black, fill=white] (.65,.96) circle (.3pt);
\filldraw[draw=black, fill=white] (.75,.05) circle (.3pt);
\end{tikzpicture}}
\hspace{5mm}
\subfigure{
\begin{tikzpicture}[scale=4]
\draw(0,0)node[below]{\small $0$}--(.25,0)node[below]{}--(.75,0)node[below]{}--(1,0)node[below]{\small $1$}--(1,1)--(0,1)node[left]{\small $1$}--(0,0);
\draw(.15,.15)node[below]{}--(.85,.15)node[below]{}--(.85,.85)--(.15,.85)node[left]{}--(.15,.15);

\draw[dotted](.35,.15)--(.35,.85) (.5,.15)--(.5,.85)(.65,.15)--(.65,.85);

\draw[draw=none,fill=blue!40,nearly transparent]  (0,0) -- (.15,0) -- (.15,1) -- (0,1)--cycle;
\draw[draw=none, fill=blue!40,nearly transparent]  (.15,0) -- (.85,.0) -- (.85,.15) -- (.15,.15)--cycle;
\draw[draw=none, fill=blue!40,nearly transparent]  (1,0) -- (1,1) -- (.85,1) -- (.85,0)--cycle;
\draw[draw=none, fill=blue!40,nearly transparent]  (.85,.85) -- (.85,1) -- (.15,1) -- (.15,.85)--cycle;

\draw[line width=0.2mm, black] (0,0)--(.15,.2)(.15,.55)--(.35,.15)(.15,.3)--(.35,.85)(.35,.6)--(.5,.15)(.35,.4)--(.5,.8)(.5,.4)--(.65,.2)(.5,.8)--(.65,.3)(.65,.45)--(.85,.75)(.65,.2)--(.85,.5)(.85,.8)--(1,1)(.15,.7)--(.35,.3)(.15,.8)--(.35,.5)(.35,.15)--(.5,.3)(.35,.7)--(.5,.2)(.5,.35)--(.65,.15)(.5,.6)--(.65,.75)(.65,.35)--(.85,.65)(.65,.25)--(.85,.55);

\filldraw[draw=black, fill=black] (0,0) circle (.2pt);
\filldraw[draw=black, fill=white] (.15,.2) circle (.2pt);
\filldraw[draw=black, fill=black] (.15,.55) circle (.2pt);
\filldraw[draw=black, fill=black] (.15,.3) circle (.2pt);
\filldraw[draw=black, fill=black] (.35,.15) circle (.2pt);
\filldraw[draw=black, fill=black] (.35,.85) circle (.2pt);
\filldraw[draw=black, fill=black] (.5,.15) circle (.2pt);
\filldraw[draw=black, fill=black] (.5,.3) circle (.2pt);
\filldraw[draw=black, fill=black] (.5,.8) circle (.2pt);
\filldraw[draw=black, fill=white] (.65,.45) circle (.2pt);
\filldraw[draw=black, fill=black] (1,1) circle (.2pt);
\filldraw[draw=black, fill=white] (.85,.8) circle (.2pt);
\filldraw[draw=black, fill=black] (.85,.75) circle (.2pt);
\filldraw[draw=black, fill=black] (.85,.5) circle (.2pt);
\filldraw[draw=black, fill=white] (.35,.6) circle (.2pt);
\filldraw[draw=black, fill=white] (.35,.4) circle (.2pt);
\filldraw[draw=black, fill=white] (.5,.4) circle (.2pt);
\filldraw[draw=black, fill=black] (.65,.2) circle (.2pt);
\filldraw[draw=black, fill=black] (.65,.3) circle (.2pt);

\filldraw[draw=black, fill=black] (.15,.7) circle (.2pt);
\filldraw[draw=black, fill=black] (.15,.8) circle (.2pt);
\filldraw[draw=black, fill=black] (.35,.3) circle (.2pt);
\filldraw[draw=black, fill=black] (.35,.5) circle (.2pt);
\filldraw[draw=black, fill=white] (.35,.7) circle (.2pt);
\filldraw[draw=black, fill=black] (.5,.2) circle (.2pt);
\filldraw[draw=black, fill=white] (.5,.35) circle (.2pt);
\filldraw[draw=black, fill=white] (.5,.6) circle (.2pt);
\filldraw[draw=black, fill=black] (.65,.15) circle (.2pt);
\filldraw[draw=black, fill=black] (.65,.75) circle (.2pt);
\filldraw[draw=black, fill=white] (.65,.25) circle (.2pt);
\filldraw[draw=black, fill=white] (.65,.35) circle (.2pt);
\filldraw[draw=black, fill=black] (.85,.65) circle (.2pt);
\filldraw[draw=black, fill=black] (.85,.55) circle (.2pt);

\end{tikzpicture}}
\caption{On the left is an arbitrary map $T$ satisfying the above conditions. On the right we see a random map $T$ in the white box that does not satisfy (A4). By adding the branches in the blue part and rescaling, we obtain a system that does satisfy these conditions. Note that any point in the blue part (except for 0 and $1$) moves to the white part after a finite number of iterations and stays there. Hence, any invariant density will equal 0 on the blue part.}
\label{f:extendedmap}
\end{figure}
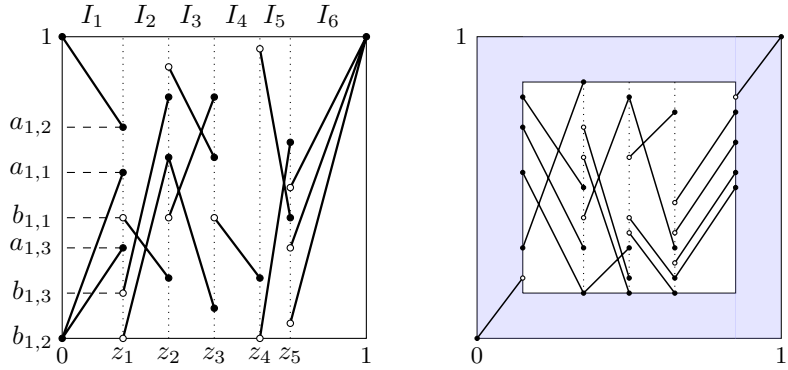

Finally, we include an assumption stating that the weighted inverse derivative cannot be 0 anywhere.

\vskip .2cm
\noindent {\bf (A5)} Assume that for any $x \in [0,1]$, the weighted inverse derivative satisfies $\sum_{j \in \Omega} \frac{p_j}{T'_j(x)}  \neq 0$. This is equivalent to assuming that for each $1 \le i \le N$,
\[ \sum_{j \in \Omega} \frac{p_j}{k_{i,j}} \neq 0.\]

\vskip .3cm
Conditions (A3) and (A5) are sufficient to get our main results, but probably not necessary. Note that (A5) is automatically fulfilled for any deterministic Lasota-Yorke map (and in particular for any deterministic piecewise linear map) and also for any random system for which on each interval $I_i$ the derivatives of all maps $T_j$ have the same sign. The last section contains an example that does not satisfy (A5) for a specific choice of $\mathbf p$. We will see that the procedure which leads to our main results still gives all invariant densities in that case. Moreover, if (A5) is not satisfied for some probability vector $\mathbf p$, then changing $\mathbf p$ slightly already lifts this restriction.

\section{A homogeneous system with a non-trivial solution}
An invariant measure reflects the dynamics of a system. For the maps $T_j$, $j \in \Omega$, considered in this article, the dynamics is determined by the orbits of the endpoints of the lines $x \mapsto k_{i,j}x + d_{i,j}$, $1 \leq i \leq N$. We start this section by defining some quantities that keep track of the possible orbits of these points.

\vskip .2cm
Let $\Omega^*$ be the set of all finite strings of elements from $\Omega$ together with the empty string $\varepsilon$. For $t \ge 0$, let $\Omega^t \subseteq \Omega^*$ denote the subset of those strings that have length $t$. So in particular, $\Omega^0 = \{ \varepsilon \}$. Let $|\omega|$ denote the length of the string $\omega$. For any string $\omega \in \Omega^*$ with $|\omega| \ge t$, we let $\omega_1^t$ denote the starting block of length $t$. For two strings $\omega, \omega' \in \Omega^*$ we simply write $\omega\omega'$ for their concatenation. Each element $\omega \in \Omega^t$ defines a possible start of an orbit of a point in $[0,1]$ by composition of maps: for $x \in [0,1]$ and $\omega = \omega_1 \cdots \omega_t \in \Omega^t$, define
\[ T_\omega (x) = T_{\omega_t} \circ T_{\omega_{t-1}} \circ \cdots \circ T_{\omega_1} (x)\]
and set $T_\varepsilon(x)=x$. For $\omega \in \Omega^*$, set $\tau_\omega (y,0)=1$ and for $1 \le t\le |\omega|$, set
\[ \tau_{\omega}(y,t) := \frac{p_{\omega_t}}{k_{i, \omega_t}}, \qquad \text{ if } \ T_{\omega^{t-1}_1} (y) \in I_{i}.
\]
Define
\begin{equation}\label{delta}
\delta_{\omega}(y, t):= \prod_{n=0 }^{t} \tau_{\omega}(y, n),
\end{equation}
so that $\delta_\omega(y,t)$ is the weighted slope of the map $T_{\omega_1^t}$ at the point $y$. Note that $\tau_{\omega}(y,t)$ and $\delta_{\omega}(y, t)$ only depend on the block $\omega_1^t$ and not on what comes after. Moreover, for a concatenation $\omega j$, given by any block $\omega$ with $|\omega|=t-1$ and any $j \in \Omega$, it holds that $ \tau_{\omega j}(y,t) = \tau_j (T_\omega(y),1)$ and $\delta_{\omega j}(y,t) = \tau_{\omega j} (y,t) \delta_\omega(y,t-1)  $. By assumption (A2) we have that for any $y \in [0,1]$,
\begin{equation} \label{q:convergence}
\begin{split}
\Big| \sum_{t \ge 0} \sum_{\omega \in \Omega^t} \delta_\omega(y,t) \Big| \le \ & 1 + \sum_{t \ge 1} \sum_{\omega \in \Omega^{t-1}} \sum_{j \in \Omega} |\delta_\omega(y,t-1)| |\tau_{\omega j} (y,t)|\\
\le \ & 1+  \sum_{t \ge 1} \sum_{\omega \in \Omega^{t-1}} |\delta_\omega(y,t-1)| \rho \le \frac1{1-\rho}.
\end{split}\end{equation}
Let $\mathbf{1}_A$ denote the characteristic function of the set $A$ and set
\[
\KI_n(y):= \sum_{t \geq 1 } \sum_{\omega\in \Omega^t}  \delta_{\omega}(y,t) \mathbf{1}_{I_n} (T_{\omega^{t-1}_1 }(y)) \qquad \text{ for } 1 \leq n \leq N.
\]
Then $\KI_n(y)$ keeps track of all the times the random orbit of $y$ visits $I_n$ and adds the corresponding weighted slopes. For $1 \leq i \leq N-1$, set $A_i:=I_1 \cup ... \cup I_i$ and $B_i:=I_{i+1} \cup ... \cup I_N$. We define 
\begin{equation}\label{KB}
\begin{split}
\KA_i(y):= \sum_{t \geq 0 } \sum_{\omega \in \Omega^t} \delta_{\omega}(y,t) \mathbf{1}_{A_i}(T_{\omega}(y)),\\
\KB_i(y):= \sum_{t \geq 0 } \sum_{\omega \in \Omega^t} \delta_{\omega}(y,t) \mathbf{1}_{B_i}(T_{\omega}(y)).
\end{split}
\end{equation}
By \eqref{q:convergence} $|\KI_n|$, $|\KA_i|$ and $|\KB_i|$ are finite for all $y \in [0,1]$. For each $1 \le n \le N$, let $S_n$ be the average inverse of the slope: 
\[ S_n := \sum_{j \in \Omega} \frac{ p_j}{k_{n,j}},\]
which is non-zero by (A5), so that $S_n^{-1}$ is well defined. The next two lemmas give some identities that we will use later.
\begin{lem}\label{l:KAKB} For each $y \in [0,1]$ and $1 \le i \le N-1$ we have
\[ \KA_i(y) = \sum_{n=1}^i S_n^{-1} \KI_n(y) \qquad \text{and} \qquad \KB_i(y) = \sum_{n=i+1}^N S_n^{-1} \KI_n(y).\]
\end{lem}

\begin{proof}
For any $1 \le n \le N$ we have
\begin{equation}\label{q:tonextt}
\begin{split}
\sum_{t \ge 0} \sum_{\omega \in \Omega^t} \delta_\omega(y,t) {\mathbf 1}_{I_n} (T_\omega(y)) & = 
\sum_{t \ge 0} \sum_{\omega \in \Omega^t}\Big(\sum_{j\in \Omega} \frac{p_j}{k_{n,j}}\Big)^{-1} \Big( \sum_{j\in \Omega} \frac{p_j}{k_{n,j}} \Big) \delta_\omega(y,t) {\mathbf 1}_{I_n} (T_\omega(y)) \\
& = \sum_{t \ge 0} \sum_{\omega \in \Omega^t}S_n^{-1} \sum_{j\in \Omega} \tau_{\omega j}(y,t+1)  \delta_\omega(y,t) {\mathbf 1}_{I_n} (T_\omega(y))\\
& = S_n^{-1} \sum_{t \ge 0} \sum_{\omega \in \Omega^{t+1}}  \delta_\omega(y,t+1) {\mathbf 1}_{I_n} (T_{\omega_1^t}(y)) = S_n^{-1} \KI_n(y).
\end{split}
\end{equation}
Putting this in the definition of $\KA_i(y)$ from \eqref{KB} gives the first part of the lemma. Using \eqref{q:tonextt}, we also get that 
\begin{equation}\label{q:kaandb}
\KA_i(y) + \KB_i (y) = \sum_{t \ge 0} \sum_{\omega \in \Omega^t} \delta_\omega(y,t) = \sum_{n=1}^N S_n^{-1} \KI_n(y).
\end{equation}
The result for $\KB_i$ follows.
\end{proof}

Define
\begin{equation}\nonumber
K_n:= S_n^{-1} -1 \qquad \text{ and } \qquad
D_n:= S_n^{-1} \bigg( \sum_{j \in \Omega} \frac{p_j}{k_{n,j}} d_{n,j}\bigg).
\end{equation}
Then
\[ \frac{D_n}{K_n}= \frac{ \sum_{j\in \Omega} \frac{p_j}{k_{n,j}}d_{n,j}}{1- \sum_{j\in \Omega} \frac{p_j}{k_{n,j}} },\]
so that we can rephrase assumption (A3) as follows: for each $1 \le i \le N$, there is an $1 \le n \le N$, such that $\frac{D_i}{K_i} \neq \frac{D_n}{K_n}$. We have the following properties for $K_n$ and $D_n$.

\begin{lem}\label{rel}
Let $y \in [0,1]$. Then 
\begin{equation}\nonumber
\sum_{n=1}^N K_n \KI_n(y)=1 \qquad \text{ and }  \qquad
-\sum_{n=1}^N D_n \KI_n(y)=y.
\end{equation}
\begin{proof}
For the first part, note that by \eqref{q:kaandb} we have
\begin{equation}\label{q:1plus}
\sum_{n=1}^N S_n^{-1} \KI_n(y) = 1+ \sum_{t \ge 1} \sum_{\omega \in \Omega^t} \delta_{\omega}(y,t) = 1 + \sum_{n=1}^N \KI_n(y).
\end{equation}
For the second part, let $1 \le i \le N$ be such that $y \in I_i$. Then for $ j \in \Omega$ we get $T_{i,j}(y)=k_{i,j} y +d_{i,j}$, and thus
$$y=  \sum_{j\in \Omega} \Big( \frac{p_j}{k_{i,j}}T_{i,j}(y) - \frac{p_j}{k_{i,j}}d_{i,j} \Big) .$$
For $t \geq 1$ and $\omega \in \Omega^*$ with $|\omega| \ge t$, set
\begin{equation}\label{q:theta}
\theta_{\omega}(y,t):= -\frac{p_{\omega_t}}{k_{n, {\omega_t} }}d_{n, {\omega_t}}  \qquad \text{ if } \ T_{\omega^{t-1}_1 }(y) \in I_n .
\end{equation}
Then 
\begin{equation}\label{star}
y= \sum_{\omega \in \Omega} \tau_{\omega}(y,1)T_{\omega}(y)+ \theta_{\omega }(y,1).
\end{equation}
Since $\tau_j(T_\omega(y),1) = \tau_{\omega j}(y,2)$ and $\theta_j(T_\omega(y),1)=\theta_{\omega j}(y,2)$, we obtain for $\omega \in \Omega$ that
\begin{equation}\label{starstar}
T_\omega (y)= \sum_{j \in \Omega} \tau_{\omega j}(y,2)T_{\omega j}(y)+ \theta_{\omega j}(y,2).
\end{equation}
Repeated application of (\ref{starstar}) in (\ref{star}), together with the definition of $\delta_\omega$ from (\ref{delta}), yields after $n$ steps,
\[
y= \sum_{t=1}^{n+1} \sum_{\omega \in \Omega^t} \delta_\omega(y,t-1) \theta_\omega(y,t) + \sum_{\omega \in \Omega^{n+1}} \delta_\omega (y,n+1) T_\omega(y).\]
From \eqref{q:convergence} we obtain that $\displaystyle \lim_{n \rightarrow \infty} \sum_{\omega \in \Omega^{n+1}} \big| \delta_{\omega}(y,n+1) T_{\omega}(y)\big|  = 0$. Hence, by (A2), \eqref{q:finiteD} and \eqref{q:convergence},
\begin{align}\label{q:yiterated}
y= & \sum_{t \geq 0} \sum_{\omega \in \Omega^{t+1}} \delta_{\omega}(y,t) \theta_{\omega}(y,t+1)\\
 = & - \sum_{n=1}^N \sum_{t \geq 0} \sum_{\omega \in \Omega^t} \delta_{\omega}(y,t) \mathbf{1}_{I_n} (T_{\omega}(y)) \bigg( \sum_{j\in \Omega} \frac{p_j}{k_{n,j}} d_{n,j} \bigg) \nonumber \\
 = & -\sum_{n=1}^N S_n^{-1} \bigg( \sum_{j\in \Omega} \frac{p_j}{k_{n,j}} d_{n,j} \bigg) \sum_{t \ge 0} \sum_{\omega \in \Omega^t} \delta_\omega(y,t) \bigg( \sum_{j\in \Omega} \frac{p_j}{k_{n,j}} \bigg) \mathbf{1}_{I_n} (T_\omega(y)) \nonumber \\
=& -\sum_{n=1}^N D_n \sum_{t \ge 0} \sum_{\omega \in \Omega^t} \delta_\omega (y,t) \Big( \sum_{j\in \Omega} \tau_{\omega j}(y,t+1) \Big)\mathbf{1}_{I_n}(T_\omega (y)) \nonumber \\ 
 = & - \sum_{n=1}^N D_n  \sum_{t \ge 0} \sum_{\omega \in \Omega^{t+1}} \delta_\omega (y,t+1)\mathbf{1}_{I_n}(T_{\omega_1^t} (y)) =  - \sum_{n=1}^N D_n \KI_n(y).
\end{align}
\end{proof}
\end{lem}

For the invariant densities, we need to keep track of the orbits of the limits from the left and from the right of each partition point. Set, for $1 \leq i \leq N-1$ and $j \in \Omega$,
$$ a_{i,j} :=k_{i,j} z_i + d_{i,j} =  \lim_{x \uparrow z_i} T_j(x), \qquad \text{and} \qquad b_{i,j}:=k_{i+1,j} z_i + d_{i+1,j} =  \lim_{x \downarrow z_i} T_j(x).
$$
See also Figure~\ref{f:extendedmap}.
\begin{defn}\label{M}
The $N \times (N-1)$-matrix $M = (\mu_{n,i})$ given by
 $$\mu_{n,i}:= \left \{
\begin{aligned}
 & \sum_{j\in \Omega} \bigg[ \frac{p_j}{k_{i,j}}+\frac{p_j}{k_{i,j}} \KI_n(a_{i,j})- \frac{p_j}{k_{i+1,j}} \KI_n(b_{i,j}) \bigg], &\mbox{for } n = i, \cr
&\sum_{j\in \Omega} \bigg[ \frac{p_j}{k_{i,j}}\KI_n(a_{i,j})-\frac{p_j}{k_{i+1,j}} - \frac{p_j}{k_{i+1,j}} \KI_n(b_{i,j}) \bigg], &\mbox{for } n = i+1, \cr
&\sum_{j\in \Omega} \bigg[ \frac{p_j}{k_{i,j}}\KI_n(a_{i,j})-\frac{p_j}{k_{i+1,j}}\KI_n(b_{i,j}) \bigg], &\mbox{else, } 
 \end{aligned}\right.
$$
is called the {\em fundamental matrix} of the random piecewise linear system $T$.
\end{defn}

Note that assumption (A2) together with the fact that $|\KI_n(y)|<\infty$ for all $y \in [0,1]$ implies that all entries of $M$ are finite. In the next section we associate invariant functions $h_\gamma$ to vectors $\gamma \in \mathbb R^{N-1}$ in the null space of $M $. Here we prove that the null space of $M$ is non-trivial.

\begin{lem}\label{sol}
The system $M \gamma =0$ admits at least one non-trivial solution.
\end{lem}

\begin{proof}
Since $M$ has dimension $N \times (N-1)$, by the Rouch\'e-Capelli Theorem the associated homogeneous system admits a non-trivial solution if and only if the rank of $M$ is at most $N-2$. Below we will give non-trivial linear dependence relations between all combinations of $N-1$ out of $N$ rows. It follows that any minor of order $N-1$ of $M$ is zero and thus that the rank of $M$ is at most $N-2$. We first show that for every $1 \le i  \le N-1$, 
\begin{equation}\label{kndn0}
\sum_{n=1}^N K_n \mu_{n,i} =0 \qquad\text{and} \qquad \sum_{n=1}^N D_n \mu_{n,i} =0.
\end{equation}
Indeed by Lemma~\ref{rel},
\begin{equation}\nonumber
\begin{split}
\sum_{n=1}^N K_n \mu_{n,i} &= \sum_{j\in \Omega} \bigg[ \frac{p_j}{k_{i,j}} K_i - \frac{p_j}{k_{i+1,j}} K_{i+1} + \frac{p_j}{k_{i,j}} \sum_{n=1}^N K_n \KI_n(a_{i,j}) -\frac{p_j}{k_{i+1,j}} \sum_{n=1}^N K_n \KI_n(b_{i,j})  \bigg]
\\&= S_i(S^{-1}_i-1) -S_{i+1}(S_{i+1}^{-1}-1) + S_i - S_{i+1} =0.
\end{split}
\end{equation}
On the other hand, by the definition of the points $a_{i,j}$ and $b_{i,j}$,
\begin{equation}\nonumber 
\begin{split}
\sum_{n=1}^N D_n \mu_{n,i} &= \sum_{j\in \Omega} \bigg[ \frac{p_j}{k_{i,j}} D_i - \frac{p_j}{k_{i+1,j}} D_{i+1} + \frac{p_j}{k_{i,j}} \sum_{n=1}^N D_n \KI_n(a_{i,j}) -\frac{p_j}{k_{i+1,j}} \sum_{n=1}^N D_n \KI_n(b_{i,j})  \bigg]
\\&= \sum_{j\in \Omega} \bigg( S_i S_i^{-1} \frac{p_j}{k_{i,j}} d_{i,j} - S_{i+1}S_{i+1}^{-1} \frac{p_j}{k_{i+1,j}} d_{i+1,j}  -  \frac{p_j}{k_{i,j}}a_{i,j} + \frac{p_j}{k_{i+1,j}}b_{i,j} \bigg)=0. 
\end{split}
\end{equation}
Consequently, for every $1 \leq l \leq N$ and every $1 \le i \le N-1$, 
\[\sum_{n=1, n \neq l}^N (D_lK_n - D_nK_l) \mu_{n,i}=0.\]
By assumption (A3) this gives non-trivial linear dependence relations between all combinations of $N-1$ out of $N$ rows, giving the result.
\end{proof}

\begin{nrem}
Note that if $S_n=0$ for some $1 \le n \le N$, then the quantities $K_n$ and $D_n$ are not well defined. In this case $\mu_{n,i} = \sum_{j \in \Omega} \frac{p_j}{k_{i,j}} \KI_n(a_{i,j}) - \frac{p_j}{k_{i+1,j}}\KI_n(b_{i,j})$ for each $1 \le i \le N-1$ and by the definition of $\KI_n$ we can write for any $y \in [0,1]$ that
\[ \begin{split}
\KI_n(y) =\ & \sum_{t \ge 1} \sum_{\omega \in \Omega^{t-1}} \sum_{j \in \Omega} \delta_{\omega} (y, t-1) \frac{p_j}{k_{n,j}} 1_{I_n} (T_{\omega_1^{t-1}}(y))\\
=\ & \sum_{t \ge 1} \sum_{\omega \in \Omega^{t-1}} \delta_{\omega} (y, t-1) 1_{I_n} (T_{\omega_1^{t-1}}(y)) S_n =0.
\end{split}\]
Hence, $\mu_{n,i}=0$ for each $i$. From this, it is clear that if $S_n=0$ for at least two indices $n$, then a non-trivial vector $\gamma$ such that $M \gamma =0$ still exists. If there is a unique $\ell$ with $S_\ell=0$, then to obtain a non-trivial solution one still needs to find suitable constants $c_n$ such that $\sum_{n=1, n \neq \ell}^N c_n \mu_{n,i}=0$ for each $i$.
\end{nrem}

Any vector $\gamma$ from the null space of $M$ satisfies the following orthogonal relations, linking $\gamma$ to the functions $\KA_i$ and $\KB_i$.
\begin{lem}\label{l:orth}
For all $1 \le i \le N-1$ we have the following orthogonal relations:
$$\gamma_i + \sum_{m=1}^{N-1} \gamma_m  \sum_{j \in \Omega}   \bigg[ \frac{p_j}{k_{m,j}} \KA_i (a_{m,j}) - \frac{p_j}{k_{m+1,j}} \KA_i (b_{m,j}) \bigg]=0; $$
and
$$\gamma_i - \sum_{m=1}^{N-1} \gamma_m \sum_{j \in \Omega}  \bigg[  \frac{p_j}{k_{m,j}} \KB_i (a_{m,j}) - \frac{p_j}{k_{m+1,j}} \KB_i (b_{m,j}) \bigg]=0.$$
\end{lem}

\begin{proof}
If $\gamma$ is a solution of the system $M\gamma=0$, then $\displaystyle \sum_{m=1}^{N-1} \gamma_m \mu_{n,m}=0$ for all $n$. Lemma \ref{l:KAKB} gives for $n=1$,
\[\begin{split}
0= S_1^{-1} \sum_{m=1}^{N-1} \gamma_m \mu_{1,m} &= S_1^{-1} \gamma_1 \sum_{j \in \Omega} \frac{p_j}{k_{1,j}} + S_1^{-1} \sum_{m=1}^{N-1} \gamma_m \sum_{j \in \Omega} \bigg( \frac{p_j}{k_{m,j}} \KI_1(a_{m,j}) -\frac{p_j}{k_{m+1,j}} \KI_1(b_{m,j})
\bigg)\\
&=\gamma_1 + \sum_{m=1}^{N-1} \gamma_m \sum_{j \in \Omega} \bigg( \frac{p_j}{k_{m,j}} \KA_1(a_{m,j}) - \frac{p_j}{k_{m+1,j}} \KA_1(b_{m,j}) \bigg) .
\end{split}\]
For $2 \le n \le N-1$ we obtain similarly
\begin{equation}\label{q:nseparate} \begin{split}
0= S_n^{-1} \sum_{m=1}^{N-1} \gamma_m \mu_{n,m} = & \ S_n^{-1} \sum_{m=1}^{N-1} \gamma_m  \sum_{j \in \Omega} \bigg( \frac{p_j}{k_{m,j}} \KI_n(a_{m,j}) - \frac{p_j}{k_{m+1,j}} \KI_n(b_{m,j})
 \bigg)\\
 & + S_n^{-1} \bigg( \gamma_n \sum_{j \in \Omega} \frac{p_j}{k_{n,j}} - \gamma_{n-1} \sum_{j \in \Omega} \frac{p_j}{k_{n,j}} \bigg)\\
 =& \ S_n^{-1} \sum_{m=1}^{N-1} \gamma_m  \sum_{j \in \Omega} \bigg( \frac{p_j}{k_{m,j}} \KI_n(a_{m,j}) - \frac{p_j}{k_{m+1,j}} \KI_n(b_{m,j})
 \bigg) + \gamma_n - \gamma_{n-1}.
\end{split}\end{equation}
Then summing over all $1 \le n \le i$ and using \eqref{q:nseparate} and Lemma \ref{l:KAKB} gives
\[\begin{split}
0= \sum_{n=1}^i S_n^{-1} \sum_{m=1}^{N-1} \gamma_m \mu_{n,m} = & \, \gamma_i + \sum_{n=1}^i S_n^{-1} \sum_{m=1}^{N-1} \gamma_m \sum_{j \in \Omega} \bigg( \frac{p_j}{k_{m,j}} \KI_n (a_{m,j}) - \frac{p_j}{k_{m+1,j}} \KI_n (b_{m,j}) \bigg)\\
=& \, \gamma_i + \sum_{m=1}^{N-1} \gamma_m \sum_{j \in \Omega} \bigg( \frac{p_j}{k_{m,j}} \KA_i(a_{m,j}) - \frac{p_j}{k_{m+1,j}} \KA_i (b_{m,j}) \bigg).
\end{split}\]
This gives the relations for $\KA_i$.

\vskip .1cm
From $\displaystyle \sum_{m=1}^{N-1} \gamma_m \mu_{n,m}=0$ for all $n$ it also follows that $\displaystyle \sum_{m=1}^{N-1} \gamma_m \sum_{n=1}^N \mu_{n,m}=0$. From this we obtain that
\[ \sum_{m=1}^{N-1} \gamma_m \sum_{j \in \Omega} \frac{p_j}{k_{m,j}} \bigg( 1 + \sum_{n=1}^N \KI_n (a_{m,j}) \bigg)= \sum_{m=1}^{N-1} \gamma_m \sum_{j \in \Omega} \frac{p_j}{k_{m+1,j}} \bigg( 1 + \sum_{n=1}^N \KI_n (b_{m,j}) \bigg).\]
Then \eqref{q:1plus} from the proof of Lemma~\ref{rel} gives that
\[ \sum_{m=1}^{N-1} \gamma_m \sum_{j \in \Omega} \frac{p_j}{k_{m,j}} \sum_{n=1}^N S_n^{-1} \KI_n(a_{m,j}) = 
\sum_{m=1}^{N-1} \gamma_m \sum_{j \in \Omega} \frac{p_j}{k_{m+1,j}} \sum_{n=1}^N S_n^{-1} \KI_n(b_{m,j}).\]
Hence, by Lemma~\ref{l:KAKB} we get for each $i$ that
\[ \sum_{m=1}^{N-1} \gamma_m \sum_{j \in \Omega} \frac{p_j}{k_{m,j}} ( \KA_i(a_{m,j}) + \KB_i(a_{m,j})) = 
\sum_{m=1}^{N-1} \gamma_m \sum_{j \in \Omega} \frac{p_j}{k_{m+1,j}} (\KA_i(b_{m,j}) + \KB_i(b_{m,j})).\]
This gives the orthogonal relations for $\KB_i$.
\end{proof}

In the proofs of our main results we only use the second part of Lemma~\ref{l:orth}, i.e., the orthogonal relations for $\KB_i$, but since we obtain the orthogonal relations for $\KA_i$ and $\KB_i$ more or less simultaneously, we have listed them both.

\section{Invariant densities for the random system $T$}

We now state our main result. For $y \in [0,1]$, define the $L^1(\lambda)$-function $L_y:[0,1]\to \mathbb R$ by
\begin{equation}\label{L}
 L_y(x)= \sum_{t \geq 0} \sum_{\omega \in  \Omega^{t}} \delta_{\omega}(y,t) \mathbf{1}_{[0,T_{\omega}(y))}(x).
 \end{equation}

\begin{thm}\label{thmm}
Let $T$ be a random piecewise linear system on the unit interval $[0,1]$ that satisfies the assumptions (A1) to (A5) from Section~\ref{s:preliminaries}. Let $M$ be the corresponding fundamental matrix and let $\gamma=(\gamma_1, \ldots ,\gamma_{N-1})^{\intercal}$ be a non-trivial solution of the system $M\gamma=0$. For each $1 \le m \le N-1$, define the function $h_m:[0,1] \to \mathbb R$ by
\begin{equation}\label{h}
h_m(x):= \sum_{\ell \in \Omega} \bigg[ \frac{p_\ell}{k_{m,\ell}} L_{a_{m,\ell}}(x)- \frac{p_\ell}{k_{m+1,\ell}} L_{b_{m,\ell}}(x) \bigg].
\end{equation}
Then a $T$-invariant function is given by
\begin{equation}\label{q:hgamma}
h_{\gamma}: [0,1] \to \mathbb R, x \mapsto \sum_{m=1}^{N-1} \gamma_m h_m(x)
\end{equation}
and $h_\gamma \neq 0$.
\end{thm}

To show that $P_T h_\gamma = h_\gamma$ $\lambda$-a.e.~we have to determine for each $x \in [0,1]$ and each branch $T_{i,j}$, whether or not $x$ has an inverse image in the branch $T_{i,j}$. Let
\[ x_{i,j} := \frac{x-d_{i,j}}{k_{i,j}}\]
be the inverse of $x$ under the map $T_{i,j}: \mathbb R \to \mathbb R$. By the definitions in (\ref{h}) and (\ref{q:hgamma}), we have to show that
\begin{equation}\label{catena}
\begin{split}
h_\gamma (x) =\ & \sum_{j \in \Omega} \sum_{i=1}^N  \frac{p_j}{|k_{i,j}|} h_{\gamma}(x_{i,j}) \mathbf{1}_{I_i}(x_{i,j})\\
 =\ &   \sum_{j \in \Omega} \sum_{i=1}^N \frac{p_j}{|k_{i,j}|} \mathbf{1}_{I_i}(x_{i,j}) \sum_{m=1}^{N-1} \gamma_m  \sum_{\ell \in \Omega} \bigg( \frac{p_\ell}{k_{m,\ell}} L_{a_{m,\ell}}(x_{i,j})- \frac{p_\ell}{k_{m+1,\ell}} L_{b_{m,\ell}}(x_{i,j}) \bigg)  .
\end{split}
\end{equation}
The parts for $L_{a_{m, \ell}}$ and $L_{b_{m,\ell}}$ behave similarly. That is why we first study
\[ \sum_{j \in \Omega} \sum_{i=1}^N \frac{p_j}{|k_{i,j}|} {\mathbf 1}_{I_i} (x_{i,j}) L_y(x_{i,j})\]
for general $y \in [0,1]$ through several lemmas. We introduce some notation to manage the long expressions. For $1 \leq i \leq N-1$, let
$$\eta_i:= \sum_{j\in \Omega} \frac{p_j ( \mathbf{1}_{(0,  \infty)}(k_{i,j}) - a_{i,j}) }{k_{i,j}} \quad \text{and} \quad \phi_i:= \sum_{j \in \Omega} \frac{p_j ( -\mathbf{1}_{(- \infty, 0)}(k_{i+1,j}) + b_{i,j}) }{k_{i+1,j}}.$$
For $y \in [0,1]$ let $1 \le n \le N$ be the index such that $y \in I_n$ and set
\begin{equation}\label{q:Cy}
C(y):=\sum_{j \in \Omega} \bigg( \sum_{i=1}^{n-1} \frac{p_j}{|k_{i,j}|} + \frac{p_j}{|k_{n,j}|}\mathbf{1}_{(-\infty ,0)}(k_{n,j}) \bigg).
\end{equation}

\begin{lem}\label{l:yandC}
Let $y \in [0,1]$. Then
\[
y = \sum_{t \ge 0} \sum_{\omega \in \Omega^{t}} \delta_{\omega}(y,t) C (T_{\omega}(y)) - \sum_{i =1}^{N-1}(\eta_i + \phi_i) \KB_i(y).
\]
\end{lem}

\begin{proof}
Let $y \in [0,1]$ be given and recall the definition of $\theta_\omega(z,t)$ from \eqref{q:theta}. If $y \in I_n$, then
\[ \begin{split}
C (y) - &\ \sum_{i=1}^{N-1} (\eta_i + \phi_i)\mathbf 1_{B_i}(y)\\
=& \  \sum_{j \in \Omega} \frac{p_j}{|k_{n,j}|} \mathbf{1}_{(-\infty,0)}(k_{n,j}) \\
& + \sum_{j \in \Omega}  \sum_{i=1}^{n-1} \bigg( \frac{p_j}{|k_{i,j}|} - \frac{p_j(\mathbf{1}_{(0, \infty)}(k_{i,j})-a_{i,j})}{k_{i,j}} - \frac{p_j(-\mathbf{1}_{(-\infty,0)}(k_{i+1,j})+b_{i,j})}{k_{i+1,j}} \bigg)\\
=&  \sum_{j \in \Omega} \bigg( -\frac{p_j}{k_{n,j}} b_{n-1,j} + \frac{p_j}{|k_{1,j}|} - \frac{p_j}{k_{1,j}} \mathbf{1}_{(0, \infty)}(k_{1,j}) + \frac{p_j}{k_{1,j}}a_{1,j} + \sum_{i=2}^{n-1} \frac{p_j}{k_{i,j}} (a_{i,j}-b_{i-1,j}) \bigg)\\
=& -\sum_{j \in \Omega} \frac{p_j}{k_{n,j}} d_{n,j} = \sum_{j \in \Omega} \theta_j(y,1), 
\end{split}\]
where we have used the assumptions from (A4) in the second to last step. So, for any $t \ge 0$ and $\omega \in \Omega^t$ we get that
\begin{equation}\label{q:comega}
C(T_\omega(y)) - \sum_{i=1}^{N-1} (\eta_i + \phi_i)\mathbf 1_{B_i}(T_\omega(y)) = \sum_{j\in \Omega} \theta_{\omega j} (y, t+1),
\end{equation}
where $\omega j$ denotes the concatenation of $\omega$ with $j \in \Omega$. Recall from the first line of \eqref{q:yiterated} that
\[ y= \sum_{t \ge 0} \sum_{\omega \in \Omega^t} \delta_{\omega}(y,t) \sum_{j \in \Omega}\theta_{\omega j} (y,t+1).\]
Combining this with \eqref{q:comega} and the definition of $\KB_i$ from \eqref{KB} then gives the result.
\end{proof}

For each $1 \leq i \leq N-1$, define the functions $E_i, F_i :[0,1] \to \mathbb R$ by
\[ \begin{split}
E_i(x):=& \sum_{j \in \Omega} \frac{p_j}{k_{i,j}} \bigg( -\mathbf{1}_{[a_{i,j},1]}(x) \mathbf{1}_{(0, \infty)}(k_{i,j})+ \mathbf{1}_{[0,a_{i,j})}(x) \mathbf{1}_{(- \infty, 0)}(k_{i,j}) \bigg),\\
F_i(x):=& \sum_{j \in \Omega}  \frac{p_j}{k_{i+1,j}} \bigg( -\mathbf{1}_{[0,b_{i,j})}(x) \mathbf{1}_{(0, \infty)}(k_{i+1,j})+ \mathbf{1}_{[b_{i,j},1]}(x) \mathbf{1}_{(- \infty, 0)}(k_{i+1,j}) \bigg)
\end{split}\]
and let $E_N,F_0: [0,1] \rightarrow \mathbb{R}$ be the zero functions. Then for each $1 \le i \le N$, we have that for Lebesgue almost every $x \in [0,1]$,
\[ E_i(x) + F_{i-1}(x) = \sum_{j \in \Omega} \frac{p_j}{|k_{i,j}|} ( \mathbf{1}_{I_i}(x_{i,j}) -1),\]
where we have used (A4) for $i=1,N$. In fact, equality holds for all but countably many points.

\begin{lem}\label{l:Ly}
For $y \in [0,1]$ we have that for Lebesgue almost every $x \in [0,1]$,
\[ \sum_{j \in \Omega} \sum_{i=1}^N \frac{p_j}{|k_{i,j}|} \mathbf{1}_{I_i}(x_{i,j}) L_y(x_{i,j}) =  \sum_{i=1}^{N-1} (E_i(x) + \eta_i + F_i(x) + \phi_i)\KB_i(y) + y + L_y(x) - {\mathbf 1}_{[0,y)}(x).\]
\end{lem}

\begin{proof}
For $y \in [0,1]$, let $1 \le n \le N$ be the index such that $y \in I_n$. By Fubini's Theorem, we get
\begin{equation}\label{q:DCT1} \sum_{j\in \Omega} \sum_{i=1}^N \frac{p_j}{|k_{i,j}|} \mathbf{1}_{I_i}(x_{i,j}) L_y(x_{i,j}) =  \sum_{t \ge 0} \sum_{\omega \in \Omega^{t}} \delta_\omega (y,t) \sum_{i=1}^N \sum_{j\in \Omega} \frac{p_j}{|k_{i,j}|} \mathbf{1}_{I_i \cap [0,T_\omega(y))}(x_{i,j}).
\end{equation}
For Lebesgue almost every $x \in [0,1]$ and for $n \neq 1$ it holds that
\begin{equation} \label{q:0y} \begin{split}
\sum_{j \in \Omega} \frac{p_j}{|k_{n,j}|}\mathbf{1}_{(-\infty,0)}(k_{n,j}) + & \sum_{j \in \Omega} \frac{p_j}{k_{n,j}} \mathbf{1}_{[0,T_j(y))}(x) + F_{n-1}(x) \\
=& \sum_{j \in \Omega} \bigg( \frac{p_j}{|k_{n,j}|}  \mathbf{1}_{(-\infty,0)}(k_{n,j})(1-\mathbf{1}_{[0,T_j(y))}(x) -\mathbf{1}_{[b_{n-1,j},1]}(x)) \\
& +  \frac{p_j}{|k_{n,j}|} \mathbf{1}_{(0,\infty)}(k_{n,j}) (\mathbf{1}_{[0,T_j(y))}(x) - \mathbf{1}_{[0,b_{n-1,j})}(x)) \bigg) \\
 =& \sum_{j\in \Omega} \frac{p_j}{|k_{n,j}|} \mathbf{1}_{I_n \cap [0,y)}(x_{n,j}).
\end{split} 
\end{equation}
Using (A4) we also get that
\begin{equation} \nonumber
\begin{split}
\sum_{j \in \Omega} \frac{p_j}{|k_{1,j}|}\mathbf{1}_{(-\infty,0)}(k_{1,j}) + & \sum_{j \in \Omega} \frac{p_j}{k_{1,j}} \mathbf{1}_{[0,T_j(y))}(x) + F_{0}(x) =  \sum_{j\in \Omega} \frac{p_j}{|k_{1,j}|} \mathbf{1}_{I_1 \cap [0,y)}(x_{1,j}),
\end{split} 
\end{equation}
so the statement from \eqref{q:0y} holds for all $1 \le n \le N$. Since $y \in I_n$ we have for Lebesgue almost every $x \in [0,1]$ that
\[
\sum_{i=1}^{N-1} (E_i(x) + F_i(x)) \mathbf{1}_{B_i}(y) = \sum_{i=1}^{n-1} \sum_{j\in \Omega} \frac{p_j}{|k_{i,j}|}(\mathbf{1}_{I_i}(x_{i,j})-1) +F_{n-1}(x).
\]
Combining this with \eqref{q:0y} and the definition of $C(y)$ from \eqref{q:Cy} we obtain that for each $y \in [0,1]$, there is a set of $x \in [0,1]$ of full Lebesgue measure, for which
\[ \begin{split}
\sum_{j\in \Omega} &\ \sum_{i=1}^N \frac{p_j}{|k_{i,j}|}  \mathbf{1}_{I_i \cap [0, y)}(x_{i,j})  
\\ = & \sum_{j \in \Omega} \sum_{i=1}^{n-1} \frac{p_j}{|k_{i,j}|} \mathbf{1}_{I_i}(x_{i,j}) + \sum_{j \in \Omega} \frac{p_j}{|k_{n,j}|}\mathbf{1}_{(-\infty,0)}(k_{n,j}) + \sum_{j\in \Omega} \frac{p_j}{k_{n,j}} \mathbf{1}_{[0,T_j(y))}(x) + F_{n-1}(x)
\\ =&  \sum_{i=1}^{N-1} (E_i(x) + F_i(x)) \mathbf{1}_{B_i}(y) + C(y) + \sum_{j \in \Omega} \tau_j(y,1) \mathbf{1}_{[0,T_j(y))}(x).
\end{split}\]
Hence, by \eqref{q:DCT1} we also have that for Lebesgue almost every $x \in [0,1]$,
\[ \begin{split}
\sum_{j\in \Omega}  \sum_{i=1}^N \frac{p_j}{|k_{i,j}|}  \mathbf{1}_{I_i}(x_{i,j}) L_y(x_{i,j}) =\ &
 \sum_{i=1}^{N-1} (E_i(x) + F_i(x)) \sum_{t \ge 0} \sum_{\omega \in \Omega^{t}} \delta_\omega (y,t) \mathbf{1}_{B_i}(T_\omega (y)) \\
 & +\sum_{t \ge 0} \sum_{\omega \in \Omega^{t}} \delta_\omega (y,t) C(T_\omega(y)) + \sum_{t \ge 1} \sum_{\omega \in \Omega^{t}} \delta_\omega (y,t) \mathbf{1}_{[0,T_{\omega}(y))}(x).
\end{split}\]
The statement now follows from the definition of $\KB_i$ from \eqref{KB} and Lemma~\ref{l:yandC}.
\end{proof}

\begin{proof}[Proof of Theorem~\ref{thmm}]
First note that for all $1 \le i \le N-1$ and all $x \in [0,1]$,
\[ E_i(x)+\eta_i = \sum_{j \in \Omega} \frac{p_j}{k_{i,j}} \bigg( \mathbf{1}_{[0,a_{i,j})}(x) - a_{i,j} \bigg)\]
and 
\[ F_i(x)+\phi_i =\sum_{j\in \Omega} \frac{p_j}{k_{i+1,j}} \bigg( -\mathbf{1}_{[0,b_{i,j})}(x) + b_{i,j} \bigg). \]
Together they give that
\[ \begin{split} \sum_{\ell \in \Omega} \bigg( \frac{p_\ell}{k_{m,\ell}}(-\mathbf{1}_{[0,a_{m,\ell})}(x) + a_{m,\ell}) - \frac{p_\ell}{k_{m+1},\ell} (-\mathbf{1}_{[0,b_{m,\ell})}(x)+b_{m,\ell}) \bigg)\\
= -(E_m(x)+\eta_m + F_m(x)+\phi_m).
\end{split}\]
Using this together with Lemma~\ref{l:Ly} and Fubini's Theorem, we get by \eqref{catena} that for Lebesgue almost every $x \in [0,1]$,
\[ \begin{split}
P_T h_\gamma(x) =\ & \sum_{m=1}^{N-1} \gamma_m   \sum_{i=1}^{N-1} (E_i(x)+\eta_i + F_i(x) + \phi_i) \sum_{\ell \in \Omega} \bigg( \frac{p_\ell}{k_{m,\ell}}\KB_i(a_{m,\ell}) - \frac{p_\ell}{k_{m+1,\ell}} \KB_i (b_{m,\ell}) \bigg)\\
& - \sum_{m=1}^{N-1} \gamma_m (E_m(x)+\eta_m + F_m(x)+\phi_m) + h_\gamma (x).
\end{split}\]
From the second part of Lemma~\ref{l:orth} we can deduce by multiplying with $E_i(x)+\eta_i + F_i(x)+\phi_i$ and summing over all $i$ that
\begin{multline*}
 \sum_{i=1}^{N-1}(E_i(x)+\eta_i +F_i(x)+\phi_i)\gamma_i \\
 = \sum_{i=1}^{N-1} (E_i(x)+\eta_i +F_i(x)+\phi_i) \sum_{m=1}^{N-1} \gamma_m  \sum_{j\in \Omega} \bigg(\frac{p_j}{k_{m,j}} \KB_i(a_{m,j}) - \frac{p_j}{k_{m+1,j}} \KB_i(b_{m,j}) \bigg).
\end{multline*}
Hence, we have obtained that $h_\gamma$ is a $T$-invariant function in $L^1(\lambda)$.

\vskip .2cm
It remains to show that $h_\gamma \neq 0$. Recall from Section~\ref{s:preliminaries} that any $T$-invariant $L^1(\lambda)$-function is of bounded variation. So, at any point $y \in [0,1]$ the limits $\lim_{x \uparrow y} h_\gamma(x)$ and $\lim_{x \downarrow y} h_\gamma(x)$ exist. Consider $1 \leq \ell \leq N-1$ and assume $z_{\ell} \in I_{\ell}$. Then for all $y \in [0,1]$, by \eqref{q:convergence} and \eqref{KB}, we obtain by the Dominated Convergence Theorem,
\begin{equation}\nonumber
\begin{split}
\lim_{x \downarrow z_\ell} L_y (x) &= \sum_{t \geq 0} \sum_{\omega \in  \Omega^{t}} \delta_{\omega}(y,t) \lim_{x \downarrow z_\ell} \mathbf{1}_{[0,T_{\omega}(y))}(x)  = \sum_{t \geq 0 } \sum_{\omega \in \Omega^t} \delta_{\omega}(y,t) \mathbf{1}_{B_{\ell}}(T_{\omega}(y)) = \KB_{\ell}(y).
\end{split}
\end{equation}
From this, Lemma \ref{l:orth} and the Dominated Convergence Theorem again we then get
\begin{equation}\label{q:limitzi}
\begin{split}
\lim_{x \downarrow z_\ell} h_{\gamma}(x) &= \sum_{m=1}^{N-1} \gamma_m \sum_{j \in \Omega} \lim_{x \downarrow z_\ell} \bigg[ \frac{p_j}{k_{m,j}} L_{a_{m,j}}(x)- \frac{p_j}{k_{m+1,j}} L_{b_{m,j}}(x) \bigg] \\
&= \sum_{m=1}^{N-1} \gamma_m \sum_{j \in \Omega} \bigg[ \frac{p_j}{k_{m,j}} \KB_{\ell}(a_{m,j})- \frac{p_j}{k_{m+1,j}} \KB_{\ell}(b_{m,j}) \bigg] \\
&= \gamma_{\ell}.
\end{split}
\end{equation}
If, on the other hand, $z_\ell \in I_{\ell+1}$, then we obtain similarly that $\lim_{x \uparrow z_\ell} L_y(x) = \KB_{\ell}(y)$ and thus that $\lim_{x \uparrow z_\ell} h_{\gamma}(x) = \gamma_\ell$. Hence, $h_{\gamma}=0$ implies $\gamma=0$. This proves the theorem.
\end{proof}

\begin{nrem}
Theorem~\ref{thmm} assigns to each solution $\gamma \neq 0$ of $M \gamma=0$ a $T$-invariant $L^1(\lambda)$-function $h_\gamma \neq 0$. From $h_\gamma$ we can get invariant densities for $T$ as follows. If $h_\gamma$ is positive or negative, then we can scale $h_\gamma$ to an invariant density function. If not, then we can write $h_\gamma = h^+ - h^-$ for two positive functions $h^+:[0,1] \to [0, \infty)$ and $h^-:[0,1] \to [0, \infty)$ and by the linearity and the positivity of $P_T$ it follows that
\[ h^+ - h^- = h_\gamma = P_T h_\gamma = P_T h^+ - P_T h^-.\]
Hence, $h^+$ and $h^-$ can both be normalised to obtain invariant densities for $T$.
\end{nrem}

\begin{nrem}\label{r:smallM}
In order to compute $h_\gamma$, one needs to compute the fundamental matrix $M$ and a vector $\gamma$ first. Lemma~\ref{sol} implies that when $N$ is small, the computation of $\gamma$ is straightforward. Indeed, for $N=2$, $M$ is the null-vector, and we can take $\gamma =1$. This is illustrated by the example of the random tent maps from Section~\ref{s:tent}. For $N=3$, it is enough to compute only one row of $M$ and take $\gamma = \begin{pmatrix} 
-\mu_{i,2}&
\mu_{i,1}
\end{pmatrix}^\intercal$. We see an illustration of this fact in Sections~\ref{sec5} and \ref{s:beta2} on random $\beta$-transformations. For larger $N$, the computation of $M$ can still be simplified by using the relations from Lemma~\ref{rel}. 
\end{nrem}

To end this section we give a small example to show that condition (A5) is not necessary for Theorem~\ref{thmm} to hold. Consider the random system with $\Omega = \{0,1\}$, $T_0(x) = 2x \pmod 1$ the doubling map, $T_1(x) = 1-T_0(x)$ and $p_0 = p_1 = \frac12$. Then $N=2$ and for both $n=1,2$ we have $S_n = \frac12 \cdot \frac12 - \frac12 \cdot \frac12 =0$. Hence $M = \begin{pmatrix} 0 & 0 \end{pmatrix}^\intercal$ and any $\gamma= \gamma_1 \in \mathbb R \setminus \{0\}$ is a non-trivial solution to $M \gamma=0$. Since all critical points of $T_0$ and $T_1$ are mapped to $0$ or $1$, the function $h_1$ from \eqref{h} will be of the form $c \cdot {\mathbf 1}_{[0,1)}$ for some $c \neq 0$ and the function $h_\gamma=\frac{\gamma}{c} \cdot {\mathbf 1}_{[0,1)}$ is indeed invariant for $T$.

\section{All possible absolutely continuous invariant measures}\label{s:alldensities}
The aim of this section is twofold. Firstly, we prove that the way $T$ is defined on the partition points $z_\ell$ does not influence the final result. In other words, the set of invariant functions we obtain from Theorem~\ref{thmm} if $z_\ell \in I_\ell$ is equal to the set of invariant functions we obtain if we choose $z_\ell \in I_{\ell+1}$. This is the content of Proposition~\ref{p:partition}. The amount of work it takes to compute the matrix $M$ and the invariant functions $h_\gamma$ depend on whether $z_\ell \in I_\ell$ or $z_\ell \in I_{\ell+1}$. Proposition~\ref{p:partition} tells us that we are free to choose the most convenient option. We shall see several examples below. Next we will use Proposition~\ref{p:partition} to prove that, under the additional assumption that all maps $T_j$ are expanding, Theorem~\ref{thmm} actually produces all absolutely continuous invariant measures of  $T$. We do this by proving in Theorem~\ref{t:alldensities} that the map $\gamma \mapsto h_\gamma$ is a bijection between the null space of $M$ and the subspace of $L^1(\lambda)$ of all $T$-invariant functions.

\begin{prop}\label{p:partition}
Let $T$ be a random system with partition $\{I_i\}_{1 \le i \le N}$ and corresponding partition points $z_0, \ldots, z_{N}$. Let $\{\hat I_i \}_{1 \le i \le N}$ be another partition of $[0,1]$ given by $z_0, \ldots, z_{N}$ and differing from $\{I_i\}_{1 \le i \le N}$ only on one or more of the points $z_1, \ldots, z_{N-1}$. Let $\hat T$ be the corresponding random system, i.e., $\hat T(x) = T(x)$ for all $x \neq z_i$, $1\le i \le N-1$. Let $\hat{M}$ be the fundamental matrix of $\hat{T}$. There is a 1-to-1 correspondence between the solutions $\gamma$ of $M \gamma=0$ and the solutions $\hat{\gamma}$ of $\hat{M} \hat{\gamma}=0$. Moreover, the functions $h_{\gamma}$ and $\hat{h}_{\hat{\gamma}}$ coincide.  
\end{prop}

\begin{proof}
First assume that there is only one point $z_\ell$ on which $\{I_i\}_{1 \le i \le N}$ and $\{\hat I_i \}_{1 \le i \le N}$ differ. We show that any column of $\hat M$ is a linear combination of columns of $M$. More precisely, we show that the $i$-th column of $\hat M$ is a linear combination of the $i$-th and the $\ell$-th column of $M$. Assume without loss of generality that $z_{\ell} \in I_{\ell}$ and therefore $z_{\ell} \in  \hat I_{\ell+1}$. This implies that $T_j(z_\ell) = a_{\ell,j}$, whereas $\hat T_j(z_\ell) = b_{\ell,j}$. This difference is reflected in the values of the quantities $\KI_n(a_{i,s})$ and $\KI_n(b_{i,s})$ appearing in the matrix $M$ in case $a_{i,s}$ or $b_{i,s}$ enters $z_\ell$ under some iteration of $T$. We will describe these changes, but first we define some quantities.

\vskip .2cm
For any $y \in \{ a_{i,j}, \, b_{i,j} \, : \, 1 \le i \le N-1, \, j \in \Omega \}$ let $\Omega_y \subseteq \Omega^*$ be the collection of paths that lead $y$ to $z_\ell$, i.e., $\omega \in \Omega_y$ if and only if there is a $0 \le t < |\omega|$, such that $T_{\omega_1^t}(y) = z_\ell$. Let
\[ \Omega_y^t := \{ \omega \in \Omega^* \, | \, \exists \, \eta \in \Omega_y \, : \, \omega= \eta_1^t, \, T_\omega (y)=z_\ell \, \text{ and } \, T_{\omega_1^s}(y) \neq z_\ell \text{ for }s<t \}.\]
Then $\Omega_y^t$ is the collection of words of length $t$ that lead $y$ to $z_\ell$ via a path that does not lead $y$ to $z_\ell$ before time $t$. We are interested in the difference between the quantities $\KI_n(y)$ and $\hKI_n(y)$ and we let $C_n^y$ denote the part that they have in common, i.e., set
\[ C_n^y := \sum_{t \ge 1} \sum_{\omega \in \Omega_y^t \cup \Omega^t \setminus \Omega_y} \delta_\omega(y,t) {\mathbf 1}_{I_n}(T_{\omega_1^{t-1}}(y)).\]
Then for $n \neq \ell$, we get
\[ \begin{split}
\KI_n(y) =\ & C_n^y + \sum_{t \ge 0} \sum_{\omega \in \Omega_y^t} \sum_{u \ge 1} \sum_{\eta \in \Omega^u} \delta_\omega (y,t) \delta_\eta(z_\ell,u) {\mathbf 1}_{I_n}(T_{\eta_1^{u-1}}(z_\ell))\\
=\ & C_n^y + \sum_{t \ge 0} \sum_{\omega \in \Omega_y^t} \sum_{u \ge 1} \sum_{\eta \in \Omega^u}  \sum_{j \in \Omega} \delta_\omega(y,t) \frac{p_j}{k_{\ell,j}} \delta_\eta (a_{\ell,j},u){\mathbf 1}_{I_n}(T_{\eta_1^{u-1}}(a_{\ell,j}))\\
=\ & C_n^y +  \sum_{t \ge 0} \sum_{\omega \in \Omega_y^t}\delta_\omega(y,t)\sum_{j \in \Omega}\frac{p_j}{k_{\ell,j}}  \KI_n(a_{\ell,j}),
\end{split}\]
and similarly, for $n=\ell$ we obtain
\[ \KI_\ell (y) = C_\ell^y +  \sum_{t \ge 0} \sum_{\omega \in \Omega_y^t}\delta_\omega(y,t)\sum_{j \in \Omega}\frac{p_j}{k_{\ell,j}} (1+ \KI_\ell(a_{\ell,j})).\]
If we set $Q(y) = \sum_{t \ge 0} \sum_{\omega \in \Omega_y^t}\delta_\omega(y,t)$ as the constant that keeps track of all the paths that lead $y$ to $z_\ell$ for the first time, then we can write
\begin{equation}\label{q:KIny} \begin{split}
\KI_n(y) = \ & C_n^y + Q(y) \sum_{j \in \Omega}\frac{p_j}{k_{\ell,j}} \KI_n(a_{\ell,j}), \, \text{for } n \neq \ell,\\
\KI_\ell(y) =\ & C_\ell^y + Q(y) \sum_{j \in \Omega}\frac{p_j}{k_{\ell,j}} (1+ \KI_\ell(a_{\ell,j})).
\end{split}\end{equation}
On the other hand, for $\hKI_n(y)$ we get
\begin{equation} \label{q:hatKIn} \begin{split}
\hKI_n(y) =\ & C_n^y + Q(y) \sum_{j \in \Omega}\frac{p_j}{k_{\ell+1,j}} \hKI_n(b_{\ell,j}), \, \text{for } n \neq \ell+1,\\
\hKI_{\ell+1}(y) =\ & C_{\ell+1}^y + Q(y) \sum_{j \in \Omega}\frac{p_j}{k_{\ell+1,j}} (1+ \hKI_{\ell+1}(b_{\ell,j})).
\end{split}\end{equation}
If $b_{\ell,j}$ does not return to $z_\ell$, then $\KI_n(b_{\ell,j}) = \hKI_n(b_{\ell,j})$. Set
\[ B : = \{ j \in \Omega \, : \, \Omega_{b_{\ell,j}} \neq \emptyset\}.\]
Then
\[ \begin{split}
\hKI_n(y) =\ & C_n^y + Q(y) \sum_{j \not \in B} \frac{p_j}{k_{\ell+1,j}} \KI_n(b_{\ell,j}) + Q(y) \sum_{j \in B}\frac{p_j}{k_{\ell+1,j}} \hKI_n(b_{\ell,j}), \, \text{for } n \neq \ell+1,\\
\hKI_{\ell+1}(y) =\ & C_{\ell+1}^y + Q(y) \sum_{j \not \in B} \frac{p_j}{k_{\ell+1,j}} (1+\KI_{\ell+1}(b_{\ell,j})) + Q(y) \sum_{j \in B}\frac{p_j}{k_{\ell+1,j}} (1+\hKI_{\ell+1}(b_{\ell,j})).
\end{split}\]
To determine the difference between $\KI_n(y)$ and $\hKI_n(y)$, we would like an expression of $\hKI_n(b_{\ell,j})$ in terms of $\KI_n(b_{\ell,j})$ for $j\in B$. Fix $n \neq \ell+1$ for a moment and set for each $j \in B$,
\[ A_j =  C_n^{b_{\ell,j}} + Q(b_{\ell,j}) \sum_{i \not \in B}\frac{p_i}{k_{\ell+1,i}} \KI_n(b_{\ell,i}).\]
Then we can find expressions of $\hKI_n(b_{\ell,j})$ in terms of the values $\KI_n(b_{\ell,i})$ by solving the following system of linear equations:
\[ \hKI_n(b_{\ell,j}) = A_j + Q(b_{\ell,j}) \sum_{i \in B} \frac{p_i}{k_{\ell+1,i}} \hKI_n (b_{\ell,i}), \quad j \in B.\]
A solution is easily computed through Cramer's method, which gives for $j \in B$
\begin{equation}
 \hKI_n (b_{\ell,j}) = \frac{\displaystyle A_j \bigg(1-\sum_{u\in B \setminus \{j\}}  Q(b_{\ell,u}) \frac{p_u}{k_{\ell+1,u}}\bigg) + Q(b_{\ell,j})  \sum_{u \in B \setminus \{ j\}} \frac{p_u}{k_{\ell+1,u}}A_u }{ \displaystyle 1- \sum_{i \in B} Q(b_{\ell,i}) \frac{p_i}{k_{\ell+1,i}} }.
\end{equation}
Set
\[ B_\ell := 1-\sum_{j \in \Omega} Q(b_{\ell,j}) \frac{p_j}{k_{\ell+1,j}}.\]
Below we will use $B_\ell^{-1}$. If $| Q(b_{\ell,j})| \leq 1$, then
\[ \Big|  \sum_{j \in \Omega} Q(b_{\ell,j}) \frac{p_j}{k_{\ell+1,j}} \Big|  \leq \sum_{j \in \Omega} | Q(b_{\ell,j})| \frac{p_j}{|k_{\ell+1,j}|} \leq \sum_{j \in \Omega} \frac{p_j}{|k_{\ell+1,j}|} \leq \rho < 1,\]
so in this case $B_\ell \neq 0$ and $B_\ell^{-1}$ is well defined. We now show that $| Q(b_{\ell,j})| \leq 1$. If $b_{\ell,j} = z_\ell$, then $\Omega_{b_{\ell,j}}^t = \emptyset $ for any $t \geq 1$, and so $Q(b_{\ell,j})=1$. If $b_{\ell,j} \neq z_\ell$, then $Q(b_{\ell,j})= \sum_{t \geq 1} \sum_{\omega \in \Omega_{b_{\ell,j}}^t}\delta_\omega(b_{\ell,j},t) $. By the expanding on average property (A2), for any $y \in I$, any $t \ge 0$ and any $\omega \in \Omega^{\mathbb{N}}$,
\begin{equation}\label{e:fatherandsons}
|\delta_{\omega}(y, t)| > \sum_{j \in \Omega} |\delta_{\omega}(y, t) \tau_j(T_{\omega_1^t}(y), 1)| = \sum_{j \in \Omega} |\delta_{\omega j}(y, t+1)|.
\end{equation}
Note that by the definition of $Q(b_{\ell,j})$ the union
\begin{equation}\label{q:disjointunion}
\bigcup_{t \ge 1} \bigcup_{\omega \in \Omega^t_{b_{\ell,j}}} [\omega] \subseteq \Omega^\mathbb N
\end{equation}
is a disjoint union of cylinder sets. Hence, by repeated application of \eqref{e:fatherandsons} we obtain for each $n \ge 1$ that
\[ \begin{split} 1 = |\delta_\epsilon (b_{\ell,j},0)| > \, & \sum_{i_1 \in \Omega} |\delta_{i_1}(b_{\ell,j},1)| = \sum_{i_1 \in \Omega_{b_{\ell,j}}} |\delta_{i_1}(b_{\ell,j},1)| + \sum_{i_1 \in \Omega_{b_{\ell,j}}^c} |\delta_{i_1}(b_{\ell,j},1)|\\
> \, & \sum_{i_1 \in \Omega_{b_{\ell,j}}} |\delta_{i_1}(b_{\ell,j},1)| + \sum_{i_1 \in \Omega_{b_{\ell,j}}^c} \sum_{i_2 \in \Omega} |\delta_{i_1i_2}(b_{\ell,j},2)|\\
= \, & \sum_{t=1}^2 \sum_{\omega \in \Omega_{b_{\ell,j}}^t} |\delta_{\omega}(b_{\ell,j},t)| + \sum_{\omega \in (\Omega_{b_{\ell,j}} \cup \Omega_{b_{\ell,j}}^2)^c} |\delta_{\omega}(b_{\ell,j},2)|\\
> \, & \cdots >  \sum_{t=1}^n \sum_{\omega \in \Omega_{b_{\ell,j}}^t} |\delta_{\omega}(b_{\ell,j},t)| + \sum_{\omega \in ( \cup_{t=1}^n\Omega_{b_{\ell,j}}^t)^c} |\delta_{\omega}(b_{\ell,j},n)|.
\end{split}\]
Since this holds for each $n$, we get $|Q(b_{\ell,j})| \le 1$ and $B_\ell \neq 0$.

\vskip .2cm
For $i \not \in B$ it holds that $\KI_n(b_{\ell,i}) = C_n^{b_{\ell,i}}$. Then by the definition of $B_\ell$, we get
\begin{equation}\label{q:shorthatKI}
\begin{split}
 \sum_{j\in B} \frac{p_j}{k_{\ell+1,j}}  \hKI_n (b_{\ell,j}) =\ &  B_\ell^{-1} \sum_{j \in B} \frac{p_j}{k_{\ell+1,j}} A_j\\
 =\ & B_\ell^{-1} \sum_{j \in B} \frac{p_j}{k_{\ell+1,j}} \Big( C^{b_{\ell,j}}_n  + Q(b_{\ell,j}) \sum_{i \not \in B} \frac{p_i}{k_{\ell+1,i}} C_n^{b_{\ell,i}}\Big)\\
 =\ & B_\ell^{-1} \sum_{j \in B} \frac{p_j}{k_{\ell+1,j}} C^{b_{\ell,j}}_n + B_\ell^{-1} (1-B_\ell) \sum_{i \not \in B}\frac{p_i}{k_{\ell+1,i}} C_n^{b_{\ell,i}}\\
 =\ & B_\ell^{-1} \sum_{j \in \Omega} \frac{p_j}{k_{\ell+1,j}} C^{b_{\ell,j}}_n - \sum_{i \not \in B}\frac{p_i}{k_{\ell+1,i}} C_n^{b_{\ell,i}}.
\end{split}\end{equation}
We obtain similar expressions for $n=\ell+1$. For each $1 \le i \le N-1$, let
\[ Q_i:= \sum_{j \in \Omega}\Big( \frac{p_j}{k_{i,j}} Q(a_{i,j}) - \frac{p_j}{k_{i+1,j}} Q(b_{i,j}) \Big).\]
We show that for each $1 \le n \le N$ and $1 \le i \le N-1$ we have
\begin{equation}\nonumber
\hat \mu_{n,i} = \mu_{n,i} - Q_iB_\ell^{-1} \mu_{n,\ell},
\end{equation}
i.e., the $i$-th column of $\hat M$ is a linear combination of the $i$-th and the $\ell$-th column of $M$. We give the proof only for $n \not \in \{ \ell, \ell+1, i, i+1\}$, since the other cases are very similar. To prove this, we first rewrite $\mu_{n,i} - Q_iB_\ell^{-1}\mu_{n,\ell}$. Therefore, note that
\begin{equation}\nonumber
\begin{split}
\sum_{j \in \Omega} \frac{p_j}{k_{\ell,j}} \KI_n(a_{\ell,j}) & - B_{\ell}^{-1} \Big( \sum_{j \in \Omega} \frac{p_j}{k_{\ell,j}} \KI_n(a_{\ell,j}) - \sum_{j \in B} \frac{p_j}{k_{\ell+1,j}} Q(b_{\ell,j}) \sum_{i \in \Omega} \frac{p_i}{k_{\ell,i}} \KI_n(a_{\ell,i}) \Big)\\
=\ & \sum_{j \in \Omega} \frac{p_j}{k_{\ell,j}}  \KI_n(a_{\ell,j})( 1-B_\ell^{-1} B_\ell)=0.
\end{split} 
\end{equation}
Then we obtain from the definition of $M$, \eqref{q:KIny} and the above equation that
\begin{equation}\nonumber \begin{split}
\mu_{n,i} - Q_iB_\ell^{-1}\mu_{n,\ell} = & \sum_{j \in \Omega} \Big( \frac{p_j}{k_{i,j}} C_n^{a_{i,j}} - \frac{p_j}{k_{i+1,j}} C_n^{b_{i,j}} \Big) + Q_i \sum_{j \in \Omega} \frac{p_j}{k_{\ell,j}} \KI_n(a_{\ell,j})\\
\ & - Q_i B_\ell^{-1} \sum_{j \in \Omega} \frac{p_j}{k_{\ell,j}} \KI_n(a_{\ell,j}) + Q_i B_\ell^{-1} \sum_{j \not \in B} \frac{p_j}{k_{\ell+1,j}} \KI_n(b_{\ell,j})\\
&+ Q_i B_\ell^{-1} \sum_{j \in B} \frac{p_j}{k_{\ell+1,j}} \Big( C_n^{b_{\ell,j}} + Q(b_{\ell,j}) \sum_{u \in \Omega} \frac{p_u}{k_{\ell,u}} \KI_n(a_{\ell,u}) \Big)\\
= \ &  \sum_{j \in \Omega} \Big( \frac{p_j}{k_{i,j}} C_n^{a_{i,j}} - \frac{p_j}{k_{i+1,j}} C_n^{b_{i,s}} \Big)  + Q_i B_\ell^{-1} \sum_{j \in \Omega} \frac{p_j}{k_{\ell+1,j}} C_n^{b_{\ell,j}}.
\end{split}
\end{equation}
For $\hat \mu_{n,i}$ we get by combining \eqref{q:hatKIn} and \eqref{q:shorthatKI} that
\begin{equation}\nonumber
\begin{split}
\hat \mu_{n,i} = & \sum_{j \in \Omega} \Big( \frac{p_j}{k_{i,j}} C_n^{a_{i,j}} + \frac{p_j}{k_{i+1,j}} C_n^{b_{i,j}} \Big) + Q_i \sum_{j \not \in B} \frac{p_j}{k_{\ell+1,j}} \KI_n(b_{\ell,j})\\
& + Q_i B_{\ell}^{-1} \sum_{j \in \Omega} \frac{p_j}{k_{\ell+1,j}} C^{b_{\ell,j}}_n - Q_i \sum_{j \not \in B}\frac{p_j}{k_{\ell+1,j}} C_n^{b_{\ell,j}} = \mu_{n,i} - Q_i B_{\ell}^{-1} \mu_{n,\ell}.\\
\end{split}
\end{equation} 

\vskip  .2cm
One now easily checks that if $\gamma = (\gamma_1, \ldots, \gamma_{N-1})^\intercal$ is a solution of $M\gamma=0$, then the vector $\hat \gamma = (\hat \gamma_1, \ldots, \hat \gamma_{N-1})^\intercal$ given by
\begin{equation}\label{q:hatgammal}
\hat \gamma_\ell = \gamma_\ell + \sum_{i=1}^{N-1} \frac{Q_i}{B_\ell - Q_\ell} \gamma_i 
\end{equation}
and $\hat \gamma_i = \gamma_i$ if $i \neq \ell$, satisfies $\hat M \hat \gamma =0$. The fact that $B_\ell - Q_\ell \neq 0$ follows in the same way as that $B_\ell \neq 0$. Hence, there is a 1-to-1 relation between the solutions $\gamma$ of $M \gamma=0$ and $\hat \gamma$ of $\hat M \hat \gamma=0$.

\vskip .2cm
It remains to prove that the functions $h_\gamma$ and $\hat h_{\hat \gamma}$ coincide. For that we need to consider the functions $L_y$. As we did for $\KI_n$, let $L^y$ denote the parts that $L_y$ and $\hat L_y$ have in common, i.e., set
\[ L^y = \sum_{t \ge 0} \sum_{\omega \in \Omega^t_y \cup \Omega^t \setminus \Omega_y} \delta_\omega (y,t) {\mathbf 1}_{[0, T_\omega(y))}.\]
Set $A : = \{ j \in \Omega \, : \, \Omega_{a_{\ell,j}} \neq \emptyset\}$. Then
\[\begin{split}
L_y =\ & L^y + Q(y) \sum_{t \ge 1} \sum_{\omega \in \Omega^t} \delta_\omega(z_\ell,t) {\mathbf 1}_{[0, \hat T_\omega(z_\ell))}\\
=\ & L^y + Q(y) \Big( \sum_{j \in \Omega} {\mathbf 1}_{[0, a_{\ell,j})} + \sum_{t \ge 1} \sum_{\omega \in \Omega^t} \frac{p_j}{k_{\ell,j}} \delta_\omega(b_{\ell,j},u) {\mathbf 1}_{[0, \hat T_\omega (a_{\ell,j}))} \Big)\\
=\ & L^y + Q(y) \sum_{j \not \in A} \frac{p_j}{k_{\ell,j}} L_{a_{\ell,j}} + Q(y) \sum_{j \in A} \frac{p_j}{k_{\ell,j}} L_{a_{\ell,j}}.
\end{split}\]
By Cramer's rule we obtain for each $j \in A$, that (compare \eqref{q:shorthatKI})
\begin{equation}\label{q:lals}
\sum_{j \in A} \frac{p_j}{k_{\ell,j}} L_{a_{\ell,j}} = (B_\ell-Q_\ell)^{-1} \sum_{j \in \Omega} \frac{p_j}{k_{\ell,j}} L^{a_{\ell,j}} -  \sum_{j \not \in A} \frac{p_j}{k_{\ell,j}} L_{a_{\ell,j}}.
\end{equation}
Similarly, we obtain that
\begin{equation} \label{q:hatLy} 
\hat L_y =  L^y + Q(y) \sum_{j \not \in B} \frac{p_j}{k_{\ell+1,j}} L_{b_{\ell,j}} + Q(y) \sum_{j \in B} \frac{p_j}{k_{\ell+1,j}} \hat L_{b_{\ell,j}}
\end{equation}
and
\begin{equation}\label{q:CramerL}
\sum_{j \in B} \frac{p_j}{k_{\ell+1,j}} \hat L_{b_{\ell,j}} = B_\ell^{-1} \sum_{j \in \Omega} \frac{p_j}{k_{\ell+1,j}} L^{b_{\ell,j}} - \sum_{j \not \in B} \frac{p_j}{k_{\ell+1,j}} L^{b_{\ell,j}}.
\end{equation}
To prove that $h_\gamma = \hat h_{\hat \gamma}$, note that on the one hand,
\[ h_\gamma =   \sum_{m=1}^{N-1} \gamma_m \sum_{j \in \Omega} \Big( \frac{p_j}{k_{m,j}} L^{a_{m,j}} - \frac{p_j}{k_{m+1,j}}L^{b_{m,j}} \Big) +  \sum_{m=1}^{N-1} \gamma_m Q_m \sum_{j \in \Omega} \frac{p_j}{k_{\ell,j}} L_{a_{\ell,j}}.\]
On the other hand, using equations \eqref{q:hatgammal}, \eqref{q:hatLy} and \eqref{q:CramerL} we obtain for $\hat h_{\hat \gamma}$ that
\[ \begin{split}
\hat h_{\hat \gamma} =\ & \sum_{m=1}^{N-1} \gamma_m \sum_{s \in \Omega} \Big( \frac{p_s}{k_{m,s}} L^{a_{m,s}} - \frac{p_s}{k_{m+1,s}} L^{b_{m,s}}\Big) + \sum_{m=1}^{N-1} \gamma_m Q_m \Big( 1+ \frac{Q_\ell}{B_\ell - Q_\ell} \Big) \sum_{s \in \Omega} \frac{p_s}{k_{\ell+1,s}} \hat L_{b_{\ell,s}}\\
 & +  \sum_{m=1}^{N-1} \gamma_m \frac{Q_m}{B_\ell - Q_\ell} \sum_{s \in \Omega} \Big( \frac{p_s}{k_{\ell,s}} L^{a_{\ell,s}} - \frac{p_s}{k_{\ell+1,s}} L^{b_{\ell,s}}\Big) \\
 =\ &  \sum_{m=1}^{N-1} \gamma_m \sum_{s \in \Omega} \Big( \frac{p_s}{k_{m,s}} L^{a_{m,s}} - \frac{p_s}{k_{m+1,s}} L^{b_{m,s}}\Big) +  \sum_{m=1}^{N-1} \gamma_m Q_m \frac{B_\ell}{B_\ell - Q_\ell} B_\ell^{-1} \sum_{s \in \Omega} \frac{p_s}{k_{\ell+1,s}} L^{b_{\ell,s}} \\
 & + \sum_{m=1}^{N-1} \gamma_m \frac{Q_m}{B_\ell - Q_\ell} \sum_{s \in \Omega} \Big( \frac{p_s}{k_{\ell,s}} L^{a_{\ell,s}} - \frac{p_s}{k_{\ell+1,s}} L^{b_{\ell,s}}\Big)  \\
 =\ & \sum_{m=1}^{N-1} \gamma_m \sum_{s \in \Omega} \Big( \frac{p_s}{k_{m,s}} L^{a_{m,s}} - \frac{p_s}{k_{m+1,s}} L^{b_{m,s}}\Big) + \sum_{m=1}^{N-1} \gamma_m \frac{Q_m}{B_\ell - Q_\ell} \sum_{s \in \Omega} \frac{p_s}{k_{\ell,s}} L^{a_{\ell,s}}.
\end{split}\]
By \eqref{q:lals} this implies that $h_\gamma = \hat h_{\hat \gamma}$.

\vskip .2cm
If the partitions $\{ I_n\}_{1 \le n \le N}$ and $\{ \hat I_n\}_{1 \le n \le N}$ differ in more than one partition point $z_\ell$, we can obtain the results from the above by changing one partition point at a time.
\end{proof}

The next lemma states that adding extra points to the set $z_0, \ldots, z_N$ does not influence the set of densities obtained from Theorem~\ref{thmm}. This lemma is one of the ingredients of the proof of Theorem~\ref{t:alldensities} below.

\begin{lem}\label{l:partition}
Let $T$ be a random system with partition $\{I_i\}_{1 \le i \le N}$ and corresponding partition points $z_0, \ldots, z_{N}$. Consider a refinement of the partition, given by adding extra points $z_1^{\dagger}, \ldots ,z_s^{\dagger}$, for some $s \in \mathbb{N}$. Let $T^{\dagger}$ be the corresponding random system, i.e., $T^{\dagger}(x) = T(x)$ for all $x \in [0,1]$, and let $M^{\dagger}$ be the fundamental matrix of $T^{\dagger}$. There is a 1-to-1 correspondence between the solutions $\gamma$ of $M \gamma=0$ and the solutions $\gamma^{\dagger}$ of $M^{\dagger} \gamma^{\dagger}=0$. Moreover, the functions $h_{\gamma}$ and $h^{\dagger}_{\gamma^{\dagger}}$ coincide.  
\end{lem}

\begin{proof}
Let $Z^{\dagger}:=\{z_1^{\dagger}, \ldots ,z_s^{\dagger}\}$. By introducing these extra points the fundamental matrix $M^\dagger$ of $T^\dagger$ becomes an $(N+s) \times (N+s-1)$ matrix. It is possible to construct this matrix from $M$ in $s$ steps
$$M \rightarrow M^{\dagger}_1 \rightarrow M^{\dagger}_2 \rightarrow \cdots \rightarrow M^{\dagger}_s=M^{\dagger} ,$$ 
by adding one of the points from $Z^{\dagger}$ to the partition of $T$ at a time. All of these steps work in exactly the same way, so it is enough to prove the result for $s=1$. Therefore, assume $ Z^{\dagger}= \{z^{\dagger}\}$. There is an $1 \le i \le N$ such that $z^{\dagger}$ splits the interval $I_i$ into two subintervals, say $I_i^L$ and $I_i^R$. By Proposition \ref{p:partition}, it is irrelevant whether $z^{\dagger} \in I_i^L$ or $z^{\dagger} \in I_i^R$. By construction, $z^\dagger$ is a continuity point of $T^{\dagger}=T$, so
$$ a_{i,j}^{\dagger}=b_{i,j}^{\dagger} =k_{i,j} z^{\dagger} + d_{i,j},$$
and for each $n$ we have
$$\sum_{j \in \Omega} \bigg[ \frac{p_j}{k_{i,j}}\KI_n(a_{i,j}^{\dagger})-\frac{p_j}{k_{i,j}}\KI_n(b_{i,j}^{\dagger}) \bigg]=0.$$
Therefore $M^{\dagger}$ has, with respect to $M$, an extra column at the $i$th position, whose entries are all zeroes except for the diagonal and subdiagonal entries, which are given by $ \sum_{j \in \Omega} \frac{p_j}{k_{i,j}}$ and $ -\sum_{j\in \Omega} \frac{p_j}{k_{i,j}}$, respectively. Moreover, the $i$th and $(i+1)$th row of $M^\dagger$ are obtained by splitting the $i$th row of $M$ into two, such that $\KI_i(a_{n,j})=\KI^\dagger_i(a_{n,j})+\KI^\dagger_{i+1}(a_{n,j})$ for all $n$, and analogously for $b_{n,j}$.

\vskip .2cm
The null space of $M^\dagger$ equals the null space of the $(N+1) \times N$ matrix $A$ obtained from $M^\dagger$ by replacing the $(i+1)$th row by the sum of the $i$th and the $(i+1)$th row. Then all the entries of the $i$th column of $A$ are 0 except for the diagonal entry, and the matrix $M$ appears as a submatrix of $A$, by deleting the $i$th column and the $i$th row. Hence, any solution $\gamma$ of $M \gamma=0$ can be transformed in a solution $\gamma^\dagger$ of $M^\dagger \gamma^\dagger=0$ by setting $\gamma^\dagger_j = \gamma_j$ for $j \neq i$ and by using the relation $\sum_{j=1}^N A_{i,j} \gamma^\dagger_j=0$ for $\gamma_i^\dagger$. This gives the first part of the lemma.

\vskip .2cm

Finally, for corresponding solutions $\gamma$ and $\gamma^\dagger$ the associated densities $h_\gamma$ and $h^\dagger_{\gamma^\dagger}$ coincide, since 
\[ \sum_{j \in \Omega} \bigg[ \frac{p_j}{k_{i,j}} L_{a_{i,j}^{\dagger}}(x)-\frac{p_j}{k_{i,j}} L_{b_{i,j}^{\dagger}}(x) \bigg]=0.
\]
\end{proof}

The next theorem says that in case all maps $T_j$ are expanding, Theorem~\ref{thmm} in fact produces all absolutely continuous invariant measures for the system $T$.

\begin{thm}\label{t:alldensities}
Let $\Omega \subseteq \mathbb N$ and let $T$ be a random piecewise linear system satisfying assumptions (A1), (A3), (A4) and (A5). Assume furthermore that $|k_{i,j}| > 1$ for each $j \in \Omega$ and $1 \le i \le N$. An $L^1(\lambda)$-function $h$ is an invariant function for the random system $T$ if and only if $h=h_\gamma$ for some solution $\gamma$ of the system $M \gamma=0$.
\end{thm}

An essential ingredient in the proof of this theorem is the extension of a result by Boyarksy, G\'ora and Islam from \cite{BGI06} given in the next lemma. \cite[Theorem 3.6]{BGI06} states that in case we have a random system consisting of two maps that are both expanding, the supports of the invariant densities of $T$ are a finite union of intervals. As the next lemma shows, this result in fact goes through for any finite or countable number of maps with only a small change in the proof. In case of piecewise linear maps, some small steps can be simplified a bit. We have included the proof for the convenience of the reader.

\begin{lem}[cf.~Lemma 3.4 and Theorem 3.6 from \cite{BGI06}]\label{l:finite}
Let $\Omega \subseteq \mathbb N$ and let $T$ be a random system of piecewise linear maps satisfying (A1) and such that for each $j \in \Omega$ the map $T_j$ is expanding, i.e., it satisfies $|k_{i,j}|>1$ for all $1\le i \le N$. If $h$ is a $T$-invariant density, then the support of $h$ is a finite union of open intervals.
\end{lem}

\begin{proof}
Let $H = \{ v_1, \ldots, v_r\}$ be the base of the subspace of $L^1(\lambda)$ of $T$-invariant functions, consisting of density functions of bounded variation, mentioned in Section~\ref{s:preliminaries}. Since any invariant function $h$ for $T$ can be written as $h= \sum_{n=1}^{r} c_n v_n$ for some constants $c_n \in \mathbb{R}$, it is enough to prove the result for elements in $H$. Therefore, let $h \in H$ and let $U:= \supp(h)$ denote the support of $h$. Since $h$ is a function of bounded variation, we can take $h$ to be lower semicontinuous and $U$ can be written as a countable union of open intervals, each separated by an interval of positive length:  $U = \bigcup_{k \ge 1} U_k$. Assume without loss of generality that $\lambda (U_{k+1}) \le \lambda(U_k)$ for each $k \ge 1$. Let $Z:=\{z_1, \ldots, z_{N-1}\}$ and let $\mathcal D$ be the set of indices $k$, such that $U_k$ contains one of the points $z \in Z$, i.e., 
\[ \mathcal D = \{ k \ge 1 \, | \, \exists \, z\in Z \, : \, z \in U_k \}.\]
We first show that $\mathcal D \neq \emptyset$ by proving that $Z \cap U_1 \neq \emptyset$. Suppose on the contrary that $U_1$ does not contain a point $z$, then for each $j \in \Omega$, $T_j (U_1)$ is an interval and since each $T_j$ is expanding, we have $\lambda (T_j(U_1)) > \lambda(U_1)$. By the property from \eqref{q:forward} that $U$ is forward invariant, we know that $T_j(U_1) \subseteq U$ for each $j$, so it must be contained in one of the intervals $U_k$. This gives a contradiction.

\vskip .2cm
Now, let $J$ be the smallest interval in the set
\[ \{ U_k \cap I_n \, : \, k \in \mathcal D, \, 1 \le n \le N \}.\]
Note that this is a finite set, since $Z$ and $\mathcal D$ are both finite. Moreover, by the above this set is not empty, so $J$ exists. Since each $U_k$ is an open interval, we have $\lambda(J)>0$. Let $\mathcal F = \{ k \ge 1 \, : \, \lambda(U_k) \ge \lambda (J) \}$, where $k$ is not necessarily in $J$, and let $S = \bigcup_{k \in \mathcal F} U_k$. Since any connected component $U_k$ of $S$ has Lebesgue measure bigger than $\lambda(J)$, $S$ is a finite union of open intervals. We first prove that $T_j (S) \subseteq S$ for any $j \in \Omega$. Let $U_k \subseteq S$ and suppose first that $k \not \in \mathcal D$. Then for each $j \in \Omega$, as above $T_j(U_k)$ is an interval with $\lambda(T_j(U_k))> \lambda(U_k) \ge \lambda(J)$. So, $T_j(U_k)$ is contained in another interval $U_i$ that satisfies $\lambda(U_i) > \lambda(J)$ and thus satisfies $U_i \subseteq S$. Hence, $T_j(U_k) \subseteq S$. If, on the other hand, $k \in \mathcal D$, then $T_j(U_k)$ consists of a finite union of intervals and since $T_j$ is expanding, the Lebesgue measure of each of these intervals exceeds $\lambda(J)$. Hence, each of the connected components of $T_j(U_k)$ is contained in some interval $U_i$ that satisfies $\lambda(U_i) > \lambda(J)$ and therefore $U_i \subseteq S$. Hence, also in this case $T_j(U_k) \subseteq S$, implying that $T_j(S) \subseteq S$ for all $j \in \Omega$.

\vskip .2cm
Obviously, $S \subseteq U$. Using the fact that $T_j(S) \subseteq S$ for all $j \in \Omega$, we will now show that $U \subseteq S$. Suppose this is not the case and let $U_s$ be the largest interval in $U \setminus S$. Since $U_k \subseteq S$ for any $k \in \mathcal D$, we have $s \not \in \mathcal D$. So, again, for each $j \in \Omega$ the set $T_j(U_s)$ is an interval with $\lambda(T_j(U_s))> \lambda (U_s)$ and hence, $T_j(U_s) \subseteq S$. Thus $U_s \subseteq T_j^{-1}(S)$ and since $U_s \not \subseteq S$, we have $U_s \subseteq T_j^{-1}(S) \setminus S$. Let $\mu_{\mathbf p}$ be the absolutely continuous $T$-invariant measure with density $h$. We show that $\mu_{\mathbf p} (T_j^{-1}(S) \setminus S)=0$. Since for each $j \in \Omega$ we have
\[ S \subseteq T_j^{-1} (T_j(S)) \subseteq T_j^{-1}(S),\]
we obtain from \eqref{q:invdensity} that
\[ \begin{split}
0 =\ & \mu_{\mathbf p} (S) - \mu_{\mathbf p} (S) = \sum_{j \in \Omega} p_j \mu_{\mathbf p} (T_j^{-1}(S)) - \sum_{j \in \Omega} p_j \mu_{\mathbf p}(S)\\
=\ & \sum_{j \in \Omega} p_j (\mu_{\mathbf p}(T^{-1}_j(S))-\mu_{\mathbf p}(S)) = \sum_{j \in \Omega} p_j \mu_{\mathbf p} (T_j^{-1}(S)\setminus S).
\end{split}\]
Since $p_j >0$ for all $j$, we have that $\mu_{\mathbf p} (T_j^{-1}(S)\setminus S)=0$ for each $j$. Hence, $\mu_{\mathbf p}(U_s)=0$, which contradicts the fact that $U_s \subseteq U$.
\end{proof}

\begin{nrem}
The article \cite{BGI06} contains an example that shows that the previous lemma is not necessarily true if we drop the assumption that all maps $T_j$ are expanding. In \cite[Example 3.7]{BGI06} the authors describe a random system $T$ using an expanding and a non-expanding map, of which for a certain probability vector $\mathbf p$ the support of the invariant density is a countable union of intervals. The fact that the supports of the elements from $H$ are finite unions of open intervals plays an essential role in the proof of Theorem~\ref{t:alldensities} as we shall see now.
\end{nrem}

\begin{proof}[Proof of Theorem~\ref{t:alldensities}]
We will show that the linear mapping from the null space of $M$ to the subspace of $L^1(\lambda)$ of all $T$-invariant functions is a linear isomorphism. Let $H=\{v_1, \ldots, v_r\}$ again be the basis of density functions of bounded variation, whose corresponding measures are ergodic, for the subspace of $T$-invariant $L^1(\lambda)$-functions mentioned in Section~\ref{s:preliminaries}. Recall that any invariant function $h$ for $T$ can be written as $h= \sum_{n=1}^{r} c_n v_n$ for some constants $c_n \in \mathbb{R}$.

\vskip .2cm
The injectivity follows from the proof of Theorem~\ref{thmm}, where we showed that $h_\gamma=0$ implies $\gamma=0$. We prove surjectivity by providing for each $h \in H$ a vector $\gamma$ such that $h_{\gamma}=h$. We will do this by altering $T$ in several steps, so that we finally obtain a system $T_U$ that has a vector $\gamma_U$ associated to it for which the corresponding density $h^U_{\gamma_U}$ vanishes outside the support $U$ of $h$. Then, using Proposition~\ref{p:partition} and Lemma~\ref{l:partition} we transform the solution $\gamma_U$ to a solution $\gamma$ for $T$ that produces the original density $h$.

\vskip .2cm
Fix $h \in H$, and let $U:= \supp(h)$. Let $Z=\{z_1,\ldots ,z_{N-1}\}$ again be the set of critical points of the system. Following \cite[Theorem 2]{Ko}, we classify the points in $Z$ as follows:
\[ \begin{split}
Z_1:= \ & \{z_i \in Z \, | \, z_i \text{ is in the interior of } U\},\\
Z_2:=\ & \{z_i \in Z \, | \, z_i \text{ is a left (right) endpoint of a subinterval of } U \text{ and }  z_i \in I_{i+1} \, (z_i \in I_i)\},\\
Z_3:=\ & \{z_i \in Z \, | \, z_i \text{ is a left (right) endpoint of a subinterval of } U \text{ and } z_i \in I_{i} \, (z_i \in I_{i+1}) \},\\
Z_4:= \ & \{z_i \in Z \, | \, z_i \text{ is an exterior point for } U\}.
\end{split}\]
We now modify the partition $\{I_i\}_{1 \le i \le N}$ on the points in $Z_3$, so that it corresponds better to the set $U$. Let $\{\hat I_i \}_{1 \le i \le N}$ be a partition of $[0,1]$ given by $z_0, \ldots, z_{N}$ and differing from $\{I_i\}_{1 \le i \le N}$ only for $z_i \in Z_3$, i.e., $z_i \in \hat{I}_i$ if and only if $z_i \notin I_i$. Let $\hat T$ be the corresponding random system, i.e., $\hat T(x) = T(x)$ for all $x \not \in Z_3$. By Proposition \ref{p:partition}, the corresponding matrices $M$ and $\hat{M}$ have vectors in their null spaces that differ only on the entries $i$ for which $z_i \in Z_3$, but such that they define the same density.

\vskip .2cm
There might be boundary points of $U$ that are not in $Z$. Let $Z^{\dagger}$ be the set of such points. From Lemma~\ref{l:finite} it follows that $U$ is a finite union of open intervals, so the set $Z^\dagger$ is finite. Consider the partition $\{ \hat I^\dagger_i \}$ given by the points in $Z \cup Z^\dagger$ and let $\hat{T}^{\dagger}$ be the system with this partition and given by $\hat{T}^{\dagger}(x)=\hat T(x) $ for all $x$. By Lemma \ref{l:partition}, the corresponding matrices $\hat{M}$ and $\hat{M}^{\dagger}$ have vectors in their null spaces that differ only on the extra entries corresponding to points $z^{\dagger} \in Z^{\dagger}$, but such that they define the same density. 

\vskip .2cm
Define a new piecewise linear random system $T_U$ by modifying $\hat T^\dagger$ outside of  $U$. To be more precise, we let $T_U(x)=\hat{T}^{\dagger}(x)$ for all $x \in U$ and on each connected component of $[0,1] \setminus U$ we assume all maps $T_{U,j}$ to be equal and onto, i.e., mapping the interval onto $[0,1]$. Recall from \eqref{q:forward} that the set $U$ is forward invariant under $T$. Then any invariant function of $T_U$ vanishes on $[0,1] \setminus U$ $\lambda$-almost everywhere, since the set of points $x \in [0,1] \setminus U$, such that $T^n(x) \in [0,1] \setminus U$ for all $n \ge 0$ is a self-similar set of Hausdorff dimension less than $1$. From Theorem \ref{thmm} we get a non-trivial solution $\gamma_U$ of $M_U \gamma_U=0$ with a corresponding function $h_U$ that vanishes on $[0,1] \setminus U$. Since $\hat T$ and $T_U$ coincide on $U$, the function $h_U$ is also invariant for $\hat T$ and hence for $T$. From the fact that $U$ is the support of one of the densities in the basis $H$ and $\supp(h_U)\subseteq U$, we then conclude that $h_U=h$, up to possibly a set of Lebesgue measure 0.

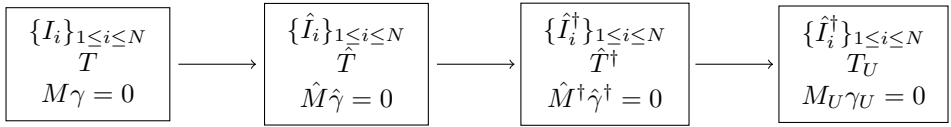
\begin{figure}
\centering
\begin{tikzpicture}[scale=1.7]
\draw[->](.8,0)--(1.2,0);
\draw[->](2.8,0)--(3.2,0);
\draw[->](4.8,0)--(5.2,0);

\node [draw] at (0,0) {
\begin{tabular}{c}
$\{I_i\}_{1 \le i \le N}$ \\  
$T$ \\
 $M\gamma=0$
 \end{tabular}};

\node [draw] at (2,0) {
\begin{tabular}{c}
$\{\hat I_i\}_{1 \le i \le N}$ \\
$\hat T$ \\
$\hat M \hat \gamma=0$
\end{tabular}};

\node [draw] at (4,0) {
\begin{tabular}{c}
$\{\hat I_i^\dagger \}_{1 \le i \le N}$\\
$\hat T^\dagger$\\
$\hat M^\dagger \hat \gamma^\dagger=0$
\end{tabular}};

\node [draw] at (6,0) {
\begin{tabular}{c}
$\{\hat I_i^\dagger\}_{1 \le i \le N}$\\
$T_U$\\
$M_U \gamma_U=0$
\end{tabular}};

\end{tikzpicture}
\caption{The steps we take in transforming $T$ to $T_U$.}
\label{f:manyTs}
\end{figure}

\vskip .2cm
It remains to show that $ \gamma_U$ can be transformed into a vector from the null space of $M$, leading to the same density $h_U$. We first show that $\hat M^\dagger \gamma_U =0$. Note that for $z_i \in Z_4$, since $h_U$ is of bounded variation, 
\[ \lim_{x \uparrow z_i} h_U(x) = 0 = \lim_{x \downarrow z_i} h_U(x).\]
Hence, by the calculations in \eqref{q:limitzi} $\gamma_{U,i}=0$. Similarly, for $z_i \in Z_2 \cup Z_3$ we have that either $\lim_{x \uparrow z_i} h_U(x)=0$ or $\lim_{x \downarrow z_i} h_U(x)=0$, which again by the calculations in \eqref{q:limitzi} gives $\gamma_{U,i}=0$. Hence, $\gamma_{U,i}=0$ for each $i$ such that $z_i \in Z_2 \cup Z_3 \cup Z_4$. Similarly, $\gamma_{U,i}=0$ for each $i$ such that $z_i \in Z^\dagger$. In the multiplication $\hat M^\dagger \gamma_U$ the orbits of the points $a_{i,j}$ and $b_{i,j}$ which are different under $\hat T^\dagger$ and $T_U$ are multiplied by 0. Since $U$ is forward invariant, all orbits of points $a_{i,j}$ and $b_{i,j}$ corresponding to $i$ such that $z_i \in Z_1$ will stay in $U$ and will thus be equal under $\hat T^\dagger$ and $T_U$. These facts imply that also $\hat M^\dagger \gamma_U=0$ and that the corresponding invariant density for $\hat T^\dagger$ is again $h_U$.

\vskip .2cm
From Lemma~\ref{l:partition} it follows that there is a vector $\hat \gamma$ in the null space of $\hat M$ with $\hat h_{\hat \gamma} = h_U$. Finally, Proposition~\ref{p:partition} then tells us how we can modify $\hat \gamma$ to get a vector $\gamma$ in the null space of $M$ with $h_\gamma = \hat h_{\hat \gamma} = h_U =h$.
\end{proof}

\section{Examples}
In this section we apply Theorems~\ref{thmm} and \ref{t:alldensities} to various examples.

\subsection{Random tent maps.}\label{s:tent}
For any countable set of slopes $\{k_j\}_{j \in \Omega}$ with $k_j \in (0,2)$ for each $j$, consider the family $T:=\{T_j\}_{j \in \Omega}$, where each $T_j$ is a tent map of slope $k_j$, i.e., $T_j:[0,1] \rightarrow [0,1]$ is given by
$$T_{j}(x)=\begin{cases}
k_j x, & \text{if } x \in [0,1/2],\\
k_j-k_j x, & \text{if } x \in (1/2,1],
\end{cases}$$
see Figure~\ref{f:tent}(a). So, (A1) and (A4) hold.
\begin{figure}[h!]
 \centering
\subfigure[Countably many tent maps.]{
 \begin{tikzpicture}[scale=3]
\draw[white](-.2,0)--(1.2,0);
\draw(0,0)node[below]{\small $0$}--(.5,0)node[below]{\small $\frac12$}--(1,0)node[below]{\small $1$}--(1,1)--(0,1)node[left]{\small $1$}--(0,0);

\draw[dotted](.5,0)--(.5,1);

\draw[dashed](0,0)--(1,1);

\draw[line width=0.3mm, bittersweet] (0,0)--(.5,1)(.5,1)--(1,0);
\draw[line width=0.3mm, bittersweet] (0,0)--(.5,.9)(.5,.9)--(1,0);
\draw[line width=0.3mm, bittersweet] (0,0)--(.5,.8)(.5,.8)--(1,0);
\draw[line width=0.3mm, bittersweet] (0,0)--(.5,.7)(.5,.7)--(1,0);
\draw[line width=0.3mm, bittersweet] (0,0)--(.5,.6)(.5,.6)--(1,0);
\draw[line width=0.3mm, bittersweet] (0,0)--(.5,.5)(.5,.5)--(1,0);
\draw[line width=0.3mm, bittersweet] (0,0)--(.5,.4)(.5,.4)--(1,0);
\draw[line width=0.3mm, bittersweet] (0,0)--(.5,.3)(.5,.3)--(1,0);
\draw[line width=0.3mm, bittersweet] (0,0)--(.5,.2)(.5,.2)--(1,0);
\draw[line width=0.3mm, bittersweet] (0,0)--(.5,.1)(.5,.1)--(1,0);
\draw[line width=0.3mm, bittersweet] (0,0)--(.5,.15)(.5,.15)--(1,0);
\draw[line width=0.3mm, bittersweet] (0,0)--(.5,.25)(.5,.25)--(1,0);
\draw[line width=0.3mm, bittersweet] (0,0)--(.5,.35)(.5,.35)--(1,0);
\draw[line width=0.3mm, bittersweet] (0,0)--(.5,.45)(.5,.45)--(1,0);
\draw[line width=0.3mm, bittersweet] (0,0)--(.5,.55)(.5,.55)--(1,0);
\draw[line width=0.3mm, bittersweet] (0,0)--(.5,.65)(.5,.65)--(1,0);
\draw[line width=0.3mm, bittersweet] (0,0)--(.5,.75)(.5,.75)--(1,0);
\draw[line width=0.3mm, bittersweet] (0,0)--(.5,.85)(.5,.85)--(1,0);
\draw[line width=0.3mm, bittersweet] (0,0)--(.5,.95)(.5,.95)--(1,0);
\draw[line width=0.3mm, bittersweet] (0,0)--(.5,.05)(.5,.05)--(1,0);
\end{tikzpicture}}
\hspace{0.5cm}
\subfigure[Two tent maps.]{
 \begin{tikzpicture}[scale=3]
\draw(0,0)node[below]{\small $0$}--(.5,0)node[below]{\small $\frac12$}--(1,0)node[below]{\small $1$}--(1,1)--(0,1)node[left]{\small $1$}--(0,0);

\draw[dotted](.5,0)--(.5,1);

\draw[dashed](0,0)--(1,1);

\draw[line width=0.3mm, bittersweet] (0,0)--(.5,.67)(.5,.67)--(1,0);
\draw[line width=0.3mm, bittersweet] (0,0)--(.5,.5)(.5,.5)--(1,0);
\end{tikzpicture}
}
\hspace{0.5cm}
\subfigure[Linear logistic maps.]{
 \begin{tikzpicture}[scale=3]
\draw(0,0)node[below]{\small $0$}--(.5,0)node[below]{\small $\frac12$}--(1,0)node[below]{\small $1$}--(1,1)--(0,1)node[left]{\small $1$}--(0,0);

\draw[dotted](.5,0)--(.5,1);

\draw[dashed](0,0)--(1,1);

\draw[line width=0.3mm, bittersweet] (0,0)--(.5,1)(.5,1)--(1,0);
\draw[line width=0.3mm, bittersweet] (0,0)--(.5,.5)(.5,.5)--(1,0);
\end{tikzpicture}
}
\caption{Random families of tent maps.}
\label{f:tent}
\end{figure}
Let $\mathbf p=(p_j)_{j \geq 0}$ be a probability vector such that $T$ is expanding on average, i.e. $\sum_{j \in \mathbb{N}} \frac{p_j}{k_j} < 1$, so (A2) holds. One easily verifies that then conditions (A3) and (A5) hold as well. For $N=2$ set 
$$z_0=0, \quad z_1=\frac12, \quad z_2=1,$$
and $I_1=[z_0, z_1]$, $I_2=(z_1,z_2]$. Since $z_1$ is the only critical point, the fundamental matrix $M$ is the null vector. As a consequence, we can choose $\gamma= 1$, to obtain the invariant density
$$h_{\gamma}= c \sum_{j \in \Omega} \frac{2p_j}{k_j} L_{k_j/2},$$
for some normalising constant $c$. If we set for each $j \in \mathbb N$ and $w \in \Omega^t$, $t \ge 0$,
\[ \ell_{\omega,j} = \# \Big\{ 1 \le n \le t \, : \, T_{\omega_1^{n-1}} \Big( \frac{k_j}{2} \Big) \in \Big( \frac12,1 \Big] \Big\},\]
then this becomes
\begin{equation}\label{q:tentdens}
h_\gamma =c \sum_{j \in \Omega} \frac{2p_j}{k_j} \sum_{t \ge 0} \sum_{\omega \in \Omega^t} (-1)^{\ell_{\omega,j}} \prod_{n=0}^t \frac{p_{\omega_n}}{k_{\omega_n}} {\mathbf 1}_{[0, T_\omega(\frac{k_j}{2}))}.
\end{equation}
If we assume that $k_j>1$ for all $j$, then it follows from Theorem~\ref{t:alldensities} that the density from \eqref{q:tentdens} is the unique invariant density for $(T,\mathbf p)$. If we do not assume this, then we can still draw the same conclusion in case there are only finitely many maps. Namely, to satisfy the condition (A2) there has to be at least one $j$ such that $k_j>1$. The existence and uniqueness of an absolutely continuous invariant measure for the map $T_j$ is then guaranteed by the results from \cite{LaYo,LiYo}. In case the set $\{k_j\}_{j \in \mathbb{N}}$ is finite, it then follows from \cite[Corollary 7]{Pe} that there is only one invariant density for $(T, \mathbf p)$.

\vskip .2cm
In \cite{AGH} the authors considered random combinations of logistic maps. In \cite[Theorem 4.2]{AGH} they proved that the random system $\{f_0, f_1\}$ with $f_0(x)=2x(1-x)$ and $f_1(x)=4x(1-x)$ has a $\sigma$-finite absolutely continuous invariant measure that is infinite in case the map $f_0$ is chosen with probability $p_0>\frac12$. The linear analogue of this system shows a different picture. Fix $a \in (1,2]$ and consider the random system with two maps $T_0(x)= \min\{x,1-x\}$ and $T_{a,1}(x) = \min\{ax, a-ax\}$. See Figure~\ref{f:tent}(b) for an example with $a=\frac43$. For any $p \in (0,1)$, set $p_0=p$ and $p_1=1-p$ and note that $p_0+\frac{p_1}{a}<1$. The assumptions (A1)-(A5) are then met and the random system $T=\{T_0,T_{a,1}\}$ has a finite absolutely continuous invariant measure for any such $p$. A straightforward computation yields $L_{\frac12}= \frac{1}{1-p} {\mathbf 1}_{[0,\frac12)} + \frac1a L_{\frac{a}{2}}$, so that up to a normalising constant, the unique absolutely continuous invariant density is then 
\begin{equation}\label{q:tentdens2}
 h_{\gamma,a} = \frac{2p}{1-p} {\mathbf 1}_{[0,\frac12)} + \frac{2}{a} L_{\frac{a}{2}}.
 \end{equation}
In particular, for $a=2$ as shown in Figure~\ref{f:tent}(c) we get
\[ h_{\gamma,2} = (1+p) {\mathbf 1}_{[0,\frac12]} + (1-p){\mathbf 1}_{(\frac12,1]}.\]
Note that for $p=1$ we have a deterministic, non-expanding interval map that does not satisfy the requirements from \cite{Ko}. However, the limit $\lim_{p \to 1} h_{\gamma,2} = 2 \cdot {\mathbf 1}_{[0,\frac12]}$ is an invariant density for the system. On the other hand, for a fixed $p \in (0,1)$ the limit $\lim_{a \to 1}h_{\gamma,a}$ is not an absolutely continuous measure. To see this, note that $h_{\gamma,a}$ is determined by the random orbits of $\frac{a}{2}$ and that $1-\frac{a}{2} \leq T_{\omega}(\frac{a}{2}) \leq \frac{a}{2}$ for any $\omega$. Hence, by \eqref{q:tentdens2} and the definition of the $L$-functions in \eqref{L} it follows that $h_{\gamma,a}=0$ on $(\frac{a}{2}, 1]$, while on $[0,1-\frac{a}{2})$ we have $h_{\gamma,a}= v $ on $[0,1-\frac{a}{2})$ for some constant $v \in \mathbb{R}$. For any point in $x \in [0,1-\frac{a}{2})$, the random Perron-Frobenius operator from \eqref{q:PF} now yields
\[ v = h_{\gamma,a}(x) = P_T h_{\gamma,a}(x)= pv+(1-p)\frac{v}{a},\]
which holds if and only if $v=0$. It follows that for any $a \in (1,2]$ and any $p \in (0,1)$, $\supp(h_{\gamma,a}) \subseteq [1-\frac{a}{2},\frac{a}{2}]$. As a consequence $\lim_{a \to 1} h_\gamma = \delta_{\frac12}$, where $\delta_{\frac12}$ is the Dirac delta function at $\frac12$.

\subsection{A random family of $W$-shaped maps.}
Keller introduced in \cite{Kel} a family of piecewise expanding $W$-maps to study the phenomenon of instability of absolutely continuous invariant measures. Later the stability of $W$-shaped maps was studied in other papers as well, see for example \cite{LGBPE,EM12}. Here we construct a random family of $W$-shaped maps, where each element of the collection is an expanding on average random map $W_a:=\{W_{a,0}, W_{a,1}\}$ defined on the unit interval. We give an absolutely continuous invariant probability measure.

\begin{figure}[h]
 \centering
\subfigure[{\color{bluebell}$W_{4}$}, {\color{babyblueeyes}$W_{8}$} and {\color{blush}$W_{\frac83}$}]{
 \begin{tikzpicture}[scale=3]
\draw(0,0)node[below]{\small $0$}--(.25,0)node[below]{\small $\frac14$}--(.375,0)node[below]{\small $\frac38$}--(.5,0)node[below]{\small $\frac12$}--(.75,0)node[below]{\small $\frac34$}--(.625,0)node[below]{\small $\frac58$}--(.125,0)node[below]{\small $\frac18$}--(.875,0)node[below]{\small $\frac78$}--(1,0)node[below]{\small $1$}--(1,1)--(0,1)node[left]{\small $1$}--(0,.25)node[left]{\small $\frac14$}--(0,.375)node[left]{\small $\frac38$}--(0,.625)node[left]{\small $\frac58$}--(0,.125)node[left]{\small $\frac18$}--(0,.875)node[left]{\small $\frac78$}--(0,.75)node[left]{\small $\frac34$}--(0,0);

\draw[dotted](.0,.375)--(.5,.375)(0,.625)--(.5,.625);
\draw[dotted](.375,0)--(.375,1)(.625,0)--(.625,1);

\draw[dotted](.0,.125)--(.5,.125)(0,.875)--(.5,.875);
\draw[dotted](.125,0)--(.125,1)(.875,0)--(.875,1);

\draw[dotted](.0,.25)--(.5,.25)(0,.75)--(.5,.75);
\draw[dotted](.25,0)--(.25,1)(.5,0)--(.5,1)(.75,0)--(.75,1);

\draw[dashed](0,0)--(1,1);

\draw[line width=0.3mm, bluebell] (0,1)--(.25,0)(.25,0)--(.5,.75)(.5,.75)--(.75,0)(.75,0)--(1,1);
\draw[line width=0.3mm, bluebell] (0,1)--(.25,0)(.25,0)--(.5,.25)(.5,.25)--(.75,0)(.75,0)--(1,1);

\draw[line width=0.3mm, blush] (0,1)--(.375,0)(.375,0)--(.5,.625)(.5,.625)--(.625,0)(.625,0)--(1,1);
\draw[line width=0.3mm, blush] (0,1)--(.375,0)(.375,0)--(.5,.375)(.5,.375)--(.625,0)(.625,0)--(1,1);

\draw[line width=0.3mm, babyblueeyes] (0,1)--(.125,0)(.125,0)--(.5,.875)(.5,.875)--(.875,0)(.875,0)--(1,1);
\draw[line width=0.3mm, babyblueeyes] (0,1)--(.125,0)(.125,0)--(.5,.125)(.5,.125)--(.875,0)(.875,0)--(1,1);

\end{tikzpicture}}
\hspace{1cm}
\subfigure[$W_2$]{
 \begin{tikzpicture}[scale=3]
\draw(0,0)node[below]{\small $0$}--(.5,0)node[below]{\small $\frac12$}--(1,0)node[below]{\small $1$}--(1,1)--(0,1)node[left]{\small $1$}--(0,0);

\filldraw[draw=darkcyan, fill=darkcyan] (.5,.5) circle (.4pt);
\draw[dotted](.5,0)--(.5,1);

\draw[dashed](0,0)--(1,1);

\draw[line width=0.3mm, darkcyan] (0,1)--(.5,0)(.5,0)--(1,1);

\end{tikzpicture}
}
\caption{Examples of random systems $W_a$ for various values of $a$.}
\label{f:Wa}
\end{figure}

\vskip .2cm
For $a > 2$, let $\Omega = \{0,1\}$ and $N=4$. Set 
$$z_0=0, \quad  z_1=1/a, \quad z_2=1/2, \quad z_3=(a-1)/a, \quad z_4=1$$ 
and 
$$I_1=[z_0, z_1], \quad I_2=( z_1, z_2], \quad I_3=(z_2, z_3), \quad I_4=[z_3,z_4].$$ 
Let
\[ W_{a,0}(x) = \begin{cases}
1-a x, & \text{if } x \in I_1, \\
\frac{2}{a-2} x- \frac{2}{(a-2)a} , & \text{if } x \in I_2, \\
W_{a,0}(1-x), & \text{otherwise},
\end{cases}
\ \text{ and } \
W_{a,1}(x) = \begin{cases}
1-a x, & \text{if } x \in I_1,\\
\frac{2(a-1)}{a-2} x -\frac{2(a-1)}{(a-2)a} , & \text{if } x \in I_2, \\
W_{a,1}(1-x), & \text{otherwise}.
\end{cases}
\]

For $a>4$ the map $W_{a,0}$ presents two contractive branches. Let $1> p > \frac{(a-4)(a-1)}{(a-2)^2}$ be arbitrary, and let $p_{a,0}=1-p$ and $p_{a,1}=p$. With this choice of probability vector the random map $W_a$ satisfies (A1)-(A5). The fundamental matrix $M$ is given by
$$M=\left(\begin{matrix}
\frac{1-a}{a^2}-\frac{C}{a} &\frac{p_{a0}(2-a)(a-1)}{a^2}+\frac{p_{a1}(2-a)}{a^2(a-1)} &  -\frac{C}{a}+\frac{1}{a^2} \\[8pt]
C & -C & 0 \\[8pt]
0 & - C & C \\[8pt]
\frac{1}{a^2(a-1)}-\frac{C}{a(a-1)} &\frac{p_{a0}(2-a)}{a^2(a-1)}-\frac{p_{a1}(2-a)(a^2-a-1)}{a^2(a-1)^2} &  -\frac{C}{a(a-1)}+\frac{1+a-a^2}{a^2(a-1)} \\
\end{matrix}\right)
$$
for some constant $C$. Its null space consists of all vectors of the form 
$$s\begin{pmatrix} 
1&
1&
1
\end{pmatrix}^\intercal, \quad s \in \mathbb R.$$
From
$$L_{0}= \frac{1}{1-a}, \qquad L_{\frac1a}= \frac{1}{a(a-1)} + \mathbf{1}_{[0, \frac1a]} \quad \text{and} \quad L_{\frac{a-1}{a}}= -\frac{1}{a(a-1)} + \mathbf{1}_{[0, \frac{a-1}{a}]},$$
we get the invariant density 
$$h_{a,p}= c \bigg[ ((a-1)-p(a-2)) \cdot \mathbf{1}_{[0, \frac1a)}+ \mathbf{1}_{[\frac1a, \frac{a-1}{a} ]}+ \bigg(1- p\frac{a-2}{a-1} \bigg) \cdot \mathbf{1}_{(\frac{a-1}{a}, 1]} \bigg],$$
for the normalising constant 
$$c= \frac{a(a-1)}{2(a-1)^2-pa(a-2)}.$$

Theorem~\ref{t:alldensities} implies that if $a<4$, then this is the unique absolutely continuous invariant density for $W_a$. Note that
\[ \lim_{a \to 2} h_{a,p}(x)= \frac{1}{2} \mathbf{1}_{[0, 1]}(x)  + \frac{1}{2} \delta_{\frac12}(x).\]
On the other hand, for the limit map $W_2$ shown in Figure~\ref{f:Wa}(b) Lebesgue measure is the only absolutely continuous invariant measure.

\subsection{Random $\beta$-transformations.}\label{sec5}
Let $ \beta >1$ be a non-integer and use $\lfloor \beta \rfloor$ to denote the largest integer not exceeding $\beta$. A {\em $\beta$-expansion} of a real number $x \in \big[0, \frac{\lfloor \beta \rfloor}{\beta-1} \big]$ is an expression of the form $x = \sum_{n=1}^{\infty} b_n \beta^{-n}$, where $b_n \in \{0,1, \ldots, \lfloor \beta \rfloor \}$ for all $n \ge 1$. The properties of $\beta$-expansions have been thoroughly studied. One of the more striking results is that Lebesgue almost all $x  \in \big[0, \frac{\lfloor \beta \rfloor}{\beta-1} \big]$ have uncountably many different $\beta$-expansions (see \cite{EJK,Sid03,DaVr2}). In \cite{DaKr} Dajani and Kraaikamp introduced a random system that produces for each $x \in \big[0, \frac{\lfloor \beta \rfloor}{\beta-1} \big]$ all its possible $\beta$-expansions. We will define this system for $1 < \beta < 2$ for simplicity, but everything easily extends to $\beta >2$. Set 
$$z_0=0, \quad z_1=\frac1{\beta}, \quad z_2 = \frac1{\beta(\beta-1)}, \quad z_3= \frac1{\beta-1},$$
and let
\[ T_0(x) = \begin{cases}
\beta x, & \text{if } x \in [z_0,z_2],\\
\beta x -1, & \text{if } x \in (z_2,z_3],
\end{cases}
\quad \text{and} \quad T_1(x) = \begin{cases}
\beta x, & \text{if } x \in [0,z_1),\\
\beta x -1, & \text{if } x \in [z_1,z_3],
\end{cases}\]
see Figure~\ref{f:beta}. The map $T_0$ is called the {\em lazy $\beta$-transformation} and the map $T_1$ is the {\em greedy $\beta$-transformation}. We do not bother to rescale the system to the unit interval $[0,1]$, since this has no effect on the computations.

\begin{figure}[h!]
\centering
\subfigure[$T_0$]{
\begin{tikzpicture}[scale=2.7]
\draw(0,0)node[below]{\small $0$}--(.58,0)node[below]{\small $\frac{1}{\beta(\beta-1)}$}--(1,0)node[below]{\small $\frac1{\beta-1}$}--(1,1)--(0,1)node[left]{\small $\frac1{\beta-1}$}--(0,.3)node[left]{\small $\frac{2-\beta}{\beta-1}$}--(0,0);

\draw[dotted](.58,0)--(.58,1);
\draw[dashed](0,.3)--(.58,.3);

\draw[thick, purple!70!black] (0,0)--(.58,1)(.58,.3)--(1,1);

\end{tikzpicture}}
\hspace{5mm}
\subfigure[$T_1$]{
\begin{tikzpicture}[scale=2.7]
\draw(0,0)node[below]{\small $0$}--(.42,0)node[below]{\small $\frac1{\beta}$}--(1,0)node[below]{\small $\frac1{\beta-1}$}--(1,1)--(0,1)node[left]{\small $\frac1{\beta-1}$}--(0,.7)node[left]{\small $1$}--(0,0);

\draw[dotted](.42,0)--(.42,1);
\draw[dashed](0,.7)--(.42,.7);

\draw[line width=0.3mm, green!70!black] (0,0)--(.42,.7)(.42,0)--(1,1);

\end{tikzpicture}}
\hspace{5mm}
\subfigure[$T$]{
\begin{tikzpicture}[scale=2.7]
\draw(0,0)node[below]{\small $0$}--(.42,0)node[below]{\small $z_1$}--(.58,0)node[below]{\small $z_2$}--(1,0)node[below]{\small $\frac1{\beta-1}$}--(1,1)--(0,1)node[left]{\small $\frac1{\beta-1}$}--(0,.7)node[left]{\small $1$}--(0,.3)node[left]{\small $\frac{2-\beta}{\beta-1}$}--(0,0);

\draw[dotted](.42,0)--(.42,1)(.58,0)--(.58,1);
\draw[dashed](0,.7)--(.42,.7)(0,.3)--(.58,.3);

\draw[line width=0.3mm, black] (0,0)--(.58,1)(.42,0)--(1,1);

\end{tikzpicture}}
\caption{In (a) we see the lazy $\beta$-transformation $T_0$, in (b) the greedy $\beta$-transformation $T_1$ and in (c) we see them combined. Whether or not $1 > \frac{2-\beta}{\beta-1}$ depends on the chosen value of $\beta$.}
\label{f:beta}
\end{figure}
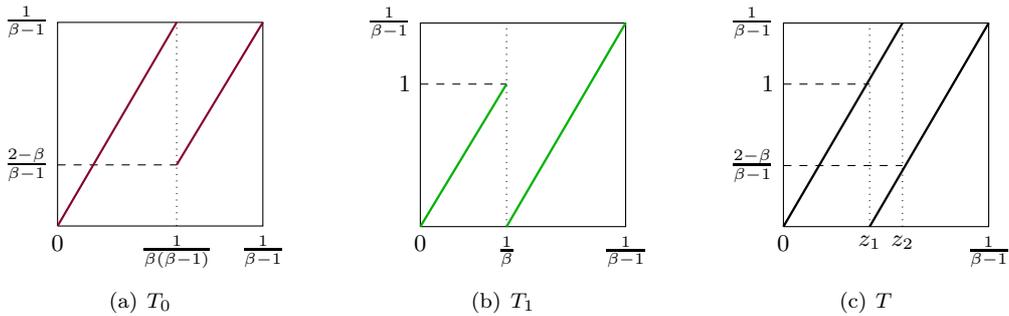

\vskip .2cm
One of the reasons why people are interested in the random $\beta$-transformation is for its relation to the infinite Bernoulli convolution, see \cite{DaVr1,DK13,Ke}. The density of the absolutely continuous invariant measures has been the subject of several papers. For a special class of values $\beta$ an explicit expression of the density of $\mu_{\mathbf p}$ was found in \cite{DaVr2} using a Markov chain. In \cite{Ke} Kempton produced an explicit formula for the invariant density for all $1 < \beta <2$ in case $p_0=p_1=\frac12$ by constructing a natural extension of the system. He states that there is a straightforward extension of this method to $\beta >2$. Recently Suzuki obtained a formula for the density of $\mu_{\mathbf p}$ for all $\beta >1$ and any $\mathbf p$ in \cite{Su}. Since the random $\beta$-transformation satisfies the assumptions (A1)-(A5) for any probability vector $\mathbf p = ( p_0 , p_1)$, we can also obtain the invariant density from Theorem~\ref{thmm}. To illustrate our method we calculate the density for $ \beta \in (1, 2)$ and $p_0=p_1=\frac12$.

\vskip .2cm
Let $\Omega = \{0,1\}$, $N=3$ and set 
$$I_1=[z_0, z_1), \quad I_2=[ z_1, z_2 ], \quad I_3=( z_2, z_3].$$ 
Define the left and right limits at each point of discontinuity:
\begin{center}
\begin{tabular}{llll}
$a_{1,0}=1$, \hspace{1cm} & $b_{1,0}=1$, \hspace{1cm}  & $a_{2,0}=\frac1{\beta-1}$, \hspace{1cm}  & $b_{2,0}=\frac{2-\beta}{\beta-1}$,\\
$a_{1,1}=1$, & $b_{1,1}=0$, & $a_{2,1}=\frac{2-\beta}{\beta-1}$, & $b_{2,1}=\frac{2-\beta}{\beta-1}$.
\end{tabular}
\end{center}
As pointed out in Remark~\ref{r:smallM}, to determine $\gamma$ it would suffice to compute only one row of $M$, but for the sake of completeness we give $M$ below. Let $\KI_n(1)=c_n$. By the symmetry of the system, for each $x \in [z_0,z_3]$ and all $(i,j) \in \{1,2,3\} \times \{0,1\}$, 
\begin{equation}\label{symmetry}
T_{i,j}(z_3-x)=z_3- T_{4-i,1-j}(x).
\end{equation}
If for any $\omega = \omega_1 \ldots \omega_t \in \{0,1\}^*$, we let $\bar \omega \in \{0,1\}^*$ denote the string $\bar \omega = (1-\omega_1)\ldots (1-\omega_t)$, then \eqref{symmetry} implies that $T_\omega(1) \in I_n$  if and only if $ T_{\bar \omega} \big(\frac{2-\beta}{\beta-1}\big) \in I_{4-n}$ and so $\KI_n\big(\frac{2-\beta}{\beta-1}\big)=c_{4-n}$.

\noindent We obtain
$$M=\left(\begin{matrix}
\frac{1}{\beta}+ \frac{1}{2 \beta} (c_1 -\frac{1}{\beta-1}) & -\frac{1}{2\beta}c_3 \\[8pt]
-\frac{1}{\beta} +\frac{1}{2 \beta}c_2 & \frac{1}{\beta}-\frac{1}{2 \beta}c_2\\[8pt]
\frac{1}{2\beta}c_3& -\frac{1}{\beta}- \frac{1}{2 \beta} (c_1 -\frac{1}{\beta-1})
\end{matrix}\right).$$
\noindent The null space consists of all vectors of the form
$$ s \begin{pmatrix} 
1 &
1
\end{pmatrix}^\intercal, \quad s \in \mathbb R.$$
From Theorem~\ref{t:alldensities} we then know that the system $T$ has a unique invariant density. We obtain
\begin{equation}\nonumber
h_{\gamma}={} \frac{c}{2 \beta} \sum_{t \geq 0} \sum_{\omega \in \{0,1\}^t} \bigg(\frac{1}{2 \beta} \bigg)^t \bigg( \mathbf{1}_{[0,T_\omega(1))}+ \mathbf{1}_{[T_\omega(\frac{2-\beta}{\beta-1}),\frac1{\beta-1}]}\bigg),
\end{equation}
for some normalising constant $c$. This matches the density found in \cite[Theorem 2.1]{Ke} except for possibly countably many points.

\vskip .2cm
If we set $p_0 \neq \frac12$, the computations are less straightforward. Nevertheless, we can obtain a nice closed formula for the density in specific instances. Let $p_0= p \in [0,1]$ be arbitrary and consider $\beta = \frac{1+\sqrt 5}{2}$, the golden mean. Then $\beta$ satisfies $\beta^2-\beta-1=0$ and the system has the nice property that $T_{2,0}(z_1) =z_2$ and $T_{2,1}(z_2) =z_1$ for $z_1= \frac1{\beta}$ and $z_2= 1$. Also note that $\frac1{\beta-1}=\beta$. This specific case has also been studied in \cite[Example 1]{DaVr2}. 
The resulting matrix $M$ is given by
$$M=\frac{\beta}{\beta^2-p(1-p)}\left(\begin{matrix}
p^2  & -p (1-p) \\[8pt]
-p  & (1-p)\\[8pt]
(1-p)p  & -(1-p)^2
\end{matrix}\right),$$
and its null space consists of all vectors of the form
$$ s \begin{pmatrix} 
1-p &
p
\end{pmatrix}^\intercal, \quad s \in \mathbb R.$$
For the functions $L_y$ we obtain $L_0=0$, $L_\beta= \beta^2$ and
\[ \begin{split}
L_{\frac1\beta} &= \frac{p^2\beta^2}{\beta^2-p(1-p)} + \frac{\beta^2}{\beta^2-p(1-p)} \mathbf 1_{[0,\frac1{\beta})}+ \frac{p\beta}{\beta^2-p(1-p)} \mathbf 1_{[0,1)},\\
L_1&= \frac{p\beta^3}{\beta^2-p(1-p)} + \frac{(1-p)\beta}{\beta^2-p(1-p)} \mathbf 1_{[0,\frac1{\beta})}+ \frac{\beta^2}{\beta^2-p(1-p)} \mathbf 1_{[0,1)}.
\end{split}\]
The unique invariant density turns out to be 
$$h_{\gamma}=  \frac{\beta^2}{1+\beta^2}\big( (1-p)\beta\cdot \mathbf{1}_{[0,\beta-1]} + \mathbf{1}_{(\beta-1, 1)} + p\beta \cdot \mathbf{1}_{[1, \beta]} \big),$$
which for $p=\frac{1}{2}$ corresponds to
$$h_{\gamma}= \frac{\beta^2}{2(1+\beta^2)}( \beta \cdot \mathbf{1}_{[0,\beta-1]} +2  \cdot \mathbf{1}_{(\beta-1, 1)} + \beta\cdot \mathbf{1}_{[1, \beta]} ).$$

\subsection{The random $(\alpha,\beta)$-transformation.}\label{s:beta2}
As an example of a system that is not everywhere expanding, but is expanding on average, we consider a random combination of the greedy $\beta$-transformation and the non-expanding $(\alpha,\beta)$-transformation introduced in \cite{DHK}. More specifically, let $0 < \alpha < 1$ and $1 < \beta < 2$ be given and 
$$z_0=0, \quad z_1= 1/\beta, \quad z_2=1.$$
Define the $(\alpha,\beta)$-transformation $T_0$ on the interval $[0,1]$ by
\begin{equation}\nonumber
T_0(x)= \left \{
\begin{aligned}
& \beta x,  &\mbox{ if } x \in [0, z_1), \cr
&\frac{\alpha}{\beta} (\beta x-1),  &\mbox{ if } x \in [z_1, z_2 ]. \end{aligned}\right.
\end{equation}
Let $T_1: [0,1] \to [0,1]$ be the greedy $\beta$-transformation again, given by $T_1(x) = \beta x \pmod 1$. For any $0<p <\frac{\alpha(\beta-1)}{\beta - \alpha}$ the random system $T$ with probability vector $\mathbf p= (p, 1-p)$ satisfies the conditions (A1), (A2), (A3) and (A5). The assumptions on the boundary points from (A4) do not hold, but this is easily solved by adding an extra interval $(z_2, z_3]$ for $z_3=\frac1{\beta-1}$ and extending $T_0$ and $T_1$ to it by setting $T_0(x) = T_1(x) = \beta x -1$. 

\vskip .2cm
This random system $T$ does not satisfy the conditions of Theorem~\ref{t:alldensities} and we can therefore not conclude directly that Theorem~\ref{thmm} produces all invariant densities for $T$. However, the set $\Omega = \{0,1\}$ is finite and the map $T_1$ is expanding with $T_1'(x)=\beta >1$ for all $x$ and therefore $T$ satisfies the conditions from \cite[Corollary 7]{Pe} on the number of ergodic components of the pseudo skew-product $R$. Since the greedy $\beta$-transformation $T_1$ has a unique absolutely continuous invariant measure, this corollary implies that also the random system $T$ has a unique invariant density. We use Theorem~\ref{thmm} to get this density. 

\vskip .2cm
Let $0 < p < \frac{\alpha(\beta-1)}{\beta - \alpha}$ be arbitrary and set 
$$I_1= [z_0, z_1), \quad I_2=[z_1,z_2], \quad I_3= (z_2, z_3].$$

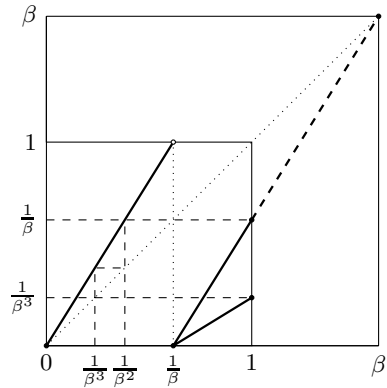
\begin{figure}[h!]
\centering
\begin{tikzpicture}[scale=2.7]
\draw(0,0)node[below]{\small $0$}--(.236,0)node[below]{\small $\frac{1}{\beta^3}$}--(.382,0)node[below]{\small $\frac{1}{\beta^2}$}--(.618,0)node[below]{\small $\frac{1}{\beta}$}--(1,0)node[below]{\small $1$}--(1,1)--(0,1)node[left]{\small $1$}--(0,1.618)node[left]{\small $\beta$}--(0,.618)node[left]{\small $\frac1{\beta}$}--(0,.236)node[left]{\small $\frac{1}{\beta^3}$}--(0,0)--(1.618,0)node[below]{\small $\beta$};

\draw[](1.618,0)--(1.618,1.618)(1,0)--(1.618,0)(0,1.618)--(1.618,1.618)(0,1)--(0,1.618);
\draw[dashed](0,.236)--(1,.236)(0,.618)--(1,.618)(.236,0)--(.236,.382)--(.382,.382)(.382,0)--(.382,.618);
\draw[dotted](.618,0)--(.618,1)(0,0)--(1.618,1.618);

\draw[line width=0.3mm, black] (0,0)--(.618,1)(.618,0)--(1,.618)(.618,0)--(1,.236);
\draw[line width=0.3mm, black, dashed] (1,.618)--(1.618,1.618);

\filldraw[draw=black, fill=black] (0,0) circle (.3pt);
\filldraw[draw=black, fill=black] (1,.236) circle (.3pt);
\filldraw[draw=black, fill=white] (.618,1) circle (.3pt);
\filldraw[draw=black, fill=black] (.618,0) circle (.3pt);
\filldraw[draw=black, fill=black] (1,.618) circle (.3pt);
\filldraw[draw=black, fill=black] (1.618,1.618) circle (.3pt);
\end{tikzpicture}
\caption{The random $(\alpha,\beta)$-transformation for $\beta = \frac{1+\sqrt 5}{2}$ and $\alpha = \frac1{\beta}$.}
\label{f:alphabeta}
\end{figure}
The left and right limits at each point of discontinuity are given by:
\vskip .1cm
\begin{center}
\begin{tabular}{llll}
$a_{1,0}=1$, \hspace{1cm} & $b_{1,0}=0$, \hspace{1cm}  & $a_{2,0}=\alpha -\frac{\alpha}{\beta}$, \hspace{1cm}  & $b_{2,0}=\beta-1$,\\
$a_{1,1}=1$, & $b_{1,1}= 0$, & $a_{2,1}=\beta-1$, & $b_{2,1}=\beta-1$.
\end{tabular}
\end{center}
\vskip .1cm
By construction, none of the points in $[0,1]$ will ever enter the interval $I_3$, therefore $\KI_3(y)=0$ for all $y\in [0,1]$. As a consequence, the last row of the $3 \times 2$ fundamental matrix $M$ is given by $\mu_{3, 1}=0$ and $\mu_{3,2} = -\frac1{\beta}$. This fact, together with the fact that we know from Lemma \ref{sol} that the null space of $M$ is non-trivial, forces the first column of $M$ to be zero, i.e., $\mu_{1,1}=\mu_{2,1}=\mu_{3,1}=0$. Hence, the null space of $M$ consists of all vectors of the form
$$ s \begin{pmatrix} 
1 &
0
\end{pmatrix}^\intercal, \quad s \in \mathbb R,$$
and the unique invariant density of the system $T$ is
\[ h_{\gamma}= \frac{c}{\beta} L_1 = \frac{c}{\beta} \sum_{t \ge 0} \sum_{\omega \in \Omega^t} \delta_\omega(1,t) {\mathbf 1}_{[0, T_\omega(1))},\]
for some normalising constant $c$. In case we choose $\beta = \frac{1+\sqrt 5}{2}$ and $\alpha = \frac1{\beta}$ as in Figure~\ref{f:alphabeta}, we can compute further to get
\[ h_{\gamma} = \frac{\beta^2}{\beta^2+1+2p} \bigg( p \beta \mathbf{1}_{[0, 1/ \beta^3]}+p \mathbf{1}_{[0, 1/ \beta^2]}+ \frac{1}{\beta} \mathbf{1}_{[0, 1/ \beta]}+\mathbf{1}_{[0, 1]} \bigg) .\]

\section{The random L\"uroth map with bounded digits}
In 1883 L\"uroth introduced in \cite{Lu} a representation of real numbers of the unit interval, as a generalisation of the decimal expansion. The standard {\em L\"uroth map} on $[0,1]$ is defined by $T_L(0)=0$ and
$$T_L(x):= n(n-1)x-(n-1), \quad \text{ if $x \in \bigg(\frac1n, \frac1{n-1}\bigg]$, $n \ge 2$}.$$
From $T_L$ we can obtain the L\"uroth expansion of any number $x \in (0,1]$ by assigning to it a sequence of positive integers $(l_n)_{n \geq 1}$, where $l_n$ is the unique integer such that $T_L^{n-1}(x) \in \big(\frac1{l_n}, \frac1{l_{n}-1}\big]$. The {\em L\"uroth expansion} of $x$ is then the expression
\begin{equation}\nonumber 
x = \sum_{n=1}^{\infty} \bigg( (l_n-1) \prod_{k=1}^n \frac{1}{l_k(l_k-1)} \bigg).
\end{equation}

The map $T_L$ was later generalised in various different ways. In \cite{KaKn1} and \cite{KaKn2}  the alternating L\"uroth map was introduced as
$$T_A(x):= 1- T_L(x).$$
This map is essentially a piecewise linear version of the Gauss map $x \mapsto \frac1x \pmod 1$, which can be used to obtain regular continued fraction expansions.
This yields for each $x \in [0,1]$ that is not a pre-image of 0 the alternating L\"uroth expansion given by
\[ x = \sum_{n=1}^{\infty} \bigg( (-1)^{n+1}a_n \prod_{k=1}^n \frac{1}{a_k(a_k-1)} \bigg),\]
where $a_n$ is the unique integer such that $T_A^{n-1}(x) \in \big(\frac1{a_n}, \frac1{a_{n}-1}\big]$. Further generalisations and ergodic properties of such maps were studied in \cite{Sal,JaVr,BaBu} for example. In \cite{BaBu} it was shown among other things that from a whole family of L\"uroth-type maps, the alternating L\"uroth map is the one with the best approximation properties.

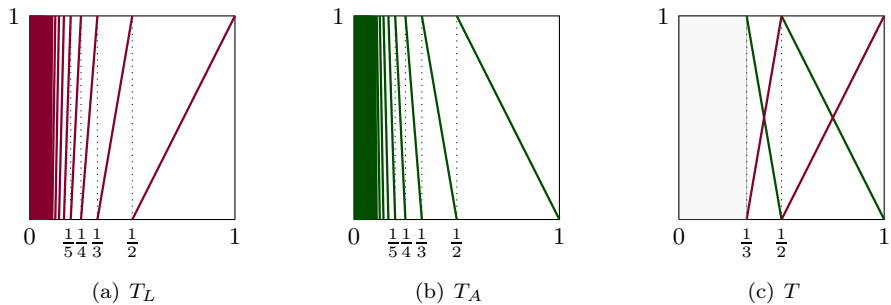
\begin{figure}[h]
 \centering
 \subfigure[$T_L$]{
 \begin{tikzpicture}[scale=2.7]
\draw(0,0)node[below]{\small $0$}--(.19,0)node[below]{\small $\frac15$}--(.255,0)node[below]{\small $\frac14$}--(.33,0)node[below]{\small $\frac13$}--(.5,0)node[below]{\small $\frac12$}--(1,0)node[below]{\small $1$}--(1,1)--(0,1)node[left]{\small $1$}--(0,0);
\draw[dotted](.2,0)--(.2,1)(.25,0)--(.25,1)(.33,0)--(.33,1)(.5,0)--(.5,1);
\draw[line width=0.3mm, purple!70!black] (.077,0)--(.083,1)(.083,0)--(.09,1)(.09,0)--(.1,1)(.1,0)--(.11,1)(.11,0)--(.125,1)(.125,0)--(.143,1)(.143,0)--(.167,1)(.167,0)--(.2,1)(.2,0)--(.25,1)(.25,0)--(.33,1)(.33,0)--(.5,1)(.5,0)--(1,1);
\draw[fill=purple!70!black, draw=purple!70!black]  (0,0) -- (.077,0) -- (.077,1) -- (0,1) -- cycle;
\end{tikzpicture}
}
\hspace{.5cm}
 \subfigure[$T_A$]{
 \begin{tikzpicture}[scale=2.7]
\draw(0,0)node[below]{\small $0$}--(.19,0)node[below]{\small $\frac15$}--(.255,0)node[below]{\small $\frac14$}--(.33,0)node[below]{\small $\frac13$}--(.5,0)node[below]{\small $\frac12$}--(1,0)node[below]{\small $1$}--(1,1)--(0,1)node[left]{\small $1$}--(0,0);
\draw[dotted](.2,0)--(.2,1)(.25,0)--(.25,1)(.33,0)--(.33,1)(.5,0)--(.5,1)(.05,0)--(.05,1);
\draw[line width=0.3mm, green!30!black]  (.077,1)--(.083,0)(.083,1)--(.09,0)(.09,1)--(.1,0)(.1,1)--(.11,0)(.11,1)--(.125,0)(.125,1)--(.143,0)(.143,1)--(.167,0)(.167,1)--(.2,0)(.2,1)--(.25,0)(.25,1)--(.33,0)(.33,1)--(.5,0)(.5,1)--(1,0);
\draw[fill=green!30!black, draw=green!30!black]  (0,0) -- (.077,0) -- (.077,1) -- (0,1) -- cycle;
\end{tikzpicture}
}
\hspace{.5cm}
 \subfigure[$T$]{
 \begin{tikzpicture}[scale=2.7]
\draw(0,0)node[below]{\small $0$}--(.33,0)node[below]{\small $\frac13$}--(.5,0)node[below]{\small $\frac12$}--(1,0)node[below]{\small $1$}--(1,1)--(0,1)node[left]{\small $1$}--(0,0);
\draw[dotted](.33,0)--(.33,1)(.5,0)--(.5,1);
\draw[line width=0.3mm, green!30!black] (.33,1)--(.5,0)(.5,1)--(1,0);
\draw[line width=0.3mm, purple!70!black] (.33,0)--(.5,1)(.5,0)--(1,1);

\draw[fill=gray!20,nearly transparent]  (0,0) -- (.33,0) -- (.33,1) -- (0,1) -- cycle;
\end{tikzpicture}

}
\caption{In (a) we see the L\"uroth map and in (b) the alternating L\"uroth map. (c) shows the open random system $T$ consisting of random combinations of $T_L$ and $T_A$ restricted to the interval $\big[\frac13,1\big]$.}
\label{f:luroth}
\end{figure}

\vskip .2cm
In this section we consider a random L\"uroth map, using $T_0:=T_L$ and $T_1:=T_A$ as its base maps. Then for each realisation of the random system $\omega \in \{0,1\}^\mathbb N$ and each $x\in [0,1]$ that is not a pre-image of 0 under the realisation $\omega$ we obtain a random L\"uroth expansion by setting for each $k \ge 0$,
\[ r_{k+1}(\omega,x) = n, \quad \text{if } T_{\omega_1^k}(x) \in \Big( \frac1n, \frac1{n-1} \Big].\]
Observe that
\[ T_{\omega_1^k}(x) = (-1)^{\omega_k} r_k(r_k-1)x + (-1)^{\omega_k-1}(r_k+ \omega_k-1).\]
If we set $s_n = \sum_{k=1}^n \omega_k$ with $s_0=0$, then we obtain the following expression for $x$:
\[ x = \sum_{n \ge 1} (-1)^{s_{n-1}}(r_n+\omega_n-1) \prod_{k=1}^n \frac1{r_k(r_k-1)}.\]
We call this expression a {\em random L\"uroth expansion} of $x$.

\vskip .2cm
Many people have considered digit properties of L\"uroth expansions, such as digit frequencies and the sizes of sets of numbers for which the digit sequence $(l_n)_{n \geq 1}$ is bounded. See for example \cite{BI09,FLMW10,SF11,MT13,GL16}. The set of points that have all L\"uroth digits bounded by some integer $a$ corresponds to the set of points that avoid the set $[0,\frac1a \big]$ under all iterations of the map $T_L$. For a deterministic system, such a set is usually a fractal no matter how large we take the upper bound $a$. In the random setting, the situation is drastically different. Fix for example $a =3$. We show below that all $x \in \big[\frac13, 1\big]$ have a random L\"uroth expansion using only digits 2 and 3. Using the density given by Theorem~\ref{thmm} we can compute the frequency of each of these digits for any typical point $x \in \big[\frac13, 1\big]$.

\vskip .2cm
Partition the interval $\big[\frac13,1 \big]$ by setting 
$$I_1=\bigg[\frac13, \frac{7}{18}\bigg], \quad I_2=\bigg(\frac{7}{18}, \frac49\bigg], \quad I_3=\bigg(\frac49, \frac12\bigg], \quad I_4=\bigg(\frac12, \frac23\bigg], \quad I_5=\bigg(\frac23, \frac56\bigg], \quad I_6=\bigg(\frac56, 1\bigg].$$ 
Let 
$$T_0(x) := \begin{cases} 
T_L(x), & \mbox{if } \displaystyle x \in I_2 \cup I_3 \cup I_5 \cup I_6, \\
\\
T_A(x), & \mbox{if } \displaystyle x \in I_1 \cup I_4,
\end{cases}
 \ \text{ and }  \
T_1(x) := \begin{cases} 
T_A(x), & \mbox{if } \displaystyle x \in I_1 \cup I_2 \cup I_4 \cup I_5, \\
\\
T_L(x), & \mbox{if } \displaystyle x \in I_3 \cup I_6.
\end{cases}$$
For $0 \leq p \leq 1$ with $p \neq \frac12$ let $p_0:=p$ and $p_1:=1-p$ and let $T$ now be the {\em random L\"uroth system with digits $2$ and $3$} defined on $\big[\frac13, 1\big]$ by setting $T(x)= T_j (x)$ with probability $p_j$, see Figure~\ref{f:randomluroth23}. Note that we have to exclude $p=\frac12$, since condition (A5) is not satisfied in that case.

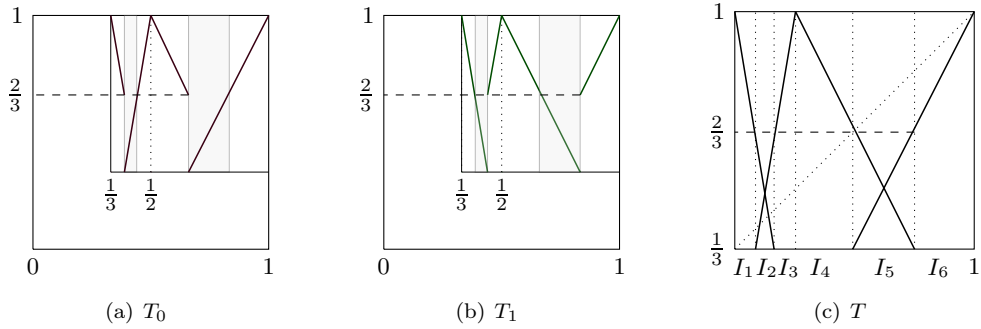
\begin{figure}[h]
\centering
\subfigure[$T_0$]{
\begin{tikzpicture}[scale=3.1]
\draw(0,0)node[below]{\small $0$}--(.33,0)node[below]{}--(.5,0)node[below]{}--(1,0)node[below]{\small $1$}--(1,1)--(0,1)node[left]{\small $1$}--(0,0);

\draw[fill=gray!20,nearly transparent]  (.66,.33) -- (.833,.33) -- (.833,1) -- (.66,1) -- cycle;
\draw[fill=gray!20,nearly transparent]  (.388,.33) -- (.44,.33) -- (.44,1) -- (.388,1) -- cycle;

\draw[dotted](.5,.33)--(.5,1);
\draw[line width=.1mm](.33,.33)--(.33,1);
\draw[line width=.1mm](.33,.33)--(1,.33);
\draw[line width=0.2mm, purple!30!black](.33,1)--(.388,.66)(.5,1)--(.66,.66);
\draw[line width=0.2mm, purple!30!black](.388,.33)--(.5,1)(.66,.33)--(1,1);

\filldraw[draw=purple!30!black, fill=purple!30!black] (.33,.33) circle (.0000001pt) node [ yshift=-0.3cm]{$\frac13$};
\filldraw[draw=purple!30!black, fill=purple!30!black] (.5,.33) circle (.0000001pt) node [ yshift=-0.3cm]{$\frac12$};

\draw[dashed] (.66,.66)--(0,.66)node[left]{$\frac23$};
\node[below] at (.355,0) {\small \color{white} $I_1$};

\end{tikzpicture}}
\hspace{5mm}
\subfigure[$T_1$]{
\begin{tikzpicture}[scale=3.1]
\draw(0,0)node[below]{\small $0$}--(.33,0)node[below]{}--(.5,0)node[below]{}--(1,0)node[below]{\small $1$}--(1,1)--(0,1)node[left]{\small $1$}--(0,0);
\draw[line width=.1mm](.33,.33)--(.33,1);
\draw[line width=.1mm](.33,.33)--(1,.33);
\draw[line width=0.2mm, green!30!black](.33,1)--(.44,.33)(.5,1)--(.833,.33);
\draw[line width=0.2mm, green!30!black](.44,.66)--(.5,1)(.833,.66)--(1,1);

\draw[fill=gray!20,nearly transparent]  (.66,.33) -- (.833,.33) -- (.833,1) -- (.66,1) -- cycle;
\draw[fill=gray!20,nearly transparent]  (.388,.33) -- (.44,.33) -- (.44,1) -- (.388,1) -- cycle;

\filldraw[draw=green!30!black, fill=green!30!black] (.33,.33) circle (.0000001pt) node [yshift=-0.3cm]{$\frac13$} ;
\filldraw[draw=green!30!black, fill=green!30!black] (.5,.33) circle (.0000001pt) node [yshift=-0.3cm]{$\frac12$} ;

\draw[dotted](.33,.33)--(.33,1)(.5,.33)--(.5,1);
\draw[dashed] (.833,.66)--(0,.66)node[left]{$\frac23$};

\node[below] at (.355,0) {\small \color{white} $I_1$};

\end{tikzpicture}}
\hspace{5mm}
\subfigure[$T$]{
\begin{tikzpicture}[scale=4.7]
\draw (.33,.33)node[left]{$\frac13$}--(.5,.33)node[below]{}--(1,.33)node[below]{\small $1$}--(1,1)--(.33,1)node[left]{\small $1$}--(.33,.33);

\draw[line width=0.2mm, black](.388,.33)--(.5,1)(.66,.33)--(1,1);
\draw[line width=0.2mm, black](.33,1)--(.44,.33)(.5,1)--(.833,.33);

\node[below] at (.355,.33) {\small $I_1$};
\node[below] at (.42,.33) {\small $I_2$};
\node[below] at (.48,.33) {\small $I_3$};
\node[below] at (.57,.33) {\small $I_4$};
\node[below] at (.75,.33) {\small $I_5$};
\node[below] at (.90,.33) {\small $I_6$};
\draw[dashed] (.833,.66)--(.33,.66)node[left]{$\frac23$};

\draw[dotted](.5,.33)--(.5,1)(.33,.33)--(1,1)(.388,.33)--(.388,1)(.44,.33)--(.44,1)(.66,.33)--(.66,1)(.833,.33)--(.833,1);

\end{tikzpicture}}
\caption{The systems $T_0$, $T_1$ and $T$ on the interval $I=[\frac13, 1]$.}
\label{f:randomluroth23}
\end{figure}

\vskip .2cm
To use Theorem~\ref{thmm}, we need to determine the orbits of all the points $a_{n,j}$ and $b_{n,j}$, which in this case are $\frac13$, $\frac23$ and $1$. One easily checks that all $\KI_n(a_{i,j})$ and $\KI_n(b_{i,j})$ are zero, except for 
$$\KI_{1}\bigg(\frac13\bigg)= -\frac{1}{6}, \qquad  \KI_{6}\bigg(\frac13\bigg)= -\frac{1}{6}, \qquad  \KI_{6}(1)=1 \quad \text{and} \quad \KI_{4}\bigg(\frac23\bigg)= -\frac13. $$
The fundamental matrix $M$ of the system is therefore given by\\
$$M=\left(\begin{matrix}
\frac{p-6}{36} & \frac{1-p}{36} & 0 & \frac{p}{12} & \frac{1-p}{12} \\[8pt]
\frac{1-2p}{6} & \frac{2p-1}{6} & 0 & 0& 0  \\[8pt]
0&-\frac{1}{6} & \frac{1}{6} & 0 & 0\\[8pt]
\frac{p}{18}& \frac{1-p}{18} & \frac{1}{2} & \frac{p-3}{6} & \frac{1-p}{6} \\[8pt]
0& 0 & 0 & \frac{1-2p}{2} & \frac{2p-1}{2}  \\[8pt]
\frac{p}{36} & \frac{1-p}{36} & \frac23 & \frac{p}{12} & -\frac{p+5}{12} 
\end{matrix}\right),$$
and its null space consists of all vectors of the form
$$s\begin{pmatrix} 
3&
3&
3&
5&
5
\end{pmatrix}^\intercal, \quad s \in \mathbb R.$$
Again this is a one-dimensional space, so by Theorem~\ref{t:alldensities} $T$ has a unique invariant density. The corresponding measure $m_{\mathbf p} \times \mu_{\mathbf p}$ is necessarily ergodic for $R$. From
$$L_{\frac13}= - \frac13, \qquad L_{\frac23}= \frac23 \cdot \mathbf{1}_{[\frac13, \frac23]} \quad \text{and} \quad L_1= 2 $$
we get the invariant density 
$$h_{\gamma} = \frac38 ( 3 \cdot \mathbf{1}_{[\frac13, \frac23]}+5 \cdot \mathbf{1}_{(\frac23, 1]}).$$

\vskip .2cm
Let $R: \{0,1\}^{\mathbb N} \times \big[ \frac13,1\big] \to \{0,1\}^{\mathbb N} \times \big[ \frac13,1\big]$ be the pseudo skew-product associated to $T$. For any point $(\omega, x) \in \{0,1\}^{\mathbb N} \times \big[ \frac13,1\big]$ the frequency of the digit 2 in its random L\"uroth expansion is given by
\[ \lim_{n \to \infty} \frac1n \sum_{k=0}^{n-1} \mathbf 1_{\{0,1\}^{\mathbb N} \times (\frac12,1]}(R^k(\omega,x)).\]
Since $m_{\mathbf p} \times \mu_{\mathbf p}$ is ergodic, by the Ergodic Theorem we have that for $m_{\mathbf p} \times \mu_{\mathbf p}$-a.e.~$(\omega,x) \in \{0,1\}^{\mathbb N} \times \big[ \frac13,1\big]$ the frequency of 2 in the associated random L\"uroth expansion is 
\[ \int_{(\frac12,1]} h_\gamma d\lambda  = \frac{13}{16},\]
giving also that the frequency of the digit 3 is $\frac3{16}$.

\vskip .2cm
Even though condition (A5) is not satisfied for $p=\frac12$, the fundamental matrix $M$ can still be computed and its null space is still given by $ s\begin{pmatrix} 
3&
3&
3&
5&
5
\end{pmatrix}^\intercal$, $s \in \mathbb R$. Moreover, the function $h_{\gamma} = \frac38 ( 3 \cdot \mathbf{1}_{[\frac13, \frac23]}+5 \cdot \mathbf{1}_{(\frac23, 1]})$ is still the unique invariant density. We believe that Theorem~\ref{thmm} and Theorem~\ref{t:alldensities} should still hold without the assumption (A5).

\vskip .2cm
\begin{nrem}
Note that our method is also capable of handling more general versions of restricted random L\"uroth maps. If, instead of considering holes of the form $\big[0, \frac13 \big)$, we would restrict the system $\{T_L, T_A\}$ to an interval $[\eta,1]$ for some $0 < \eta < 1$, then by the same arguments as above, the restricted random L\"uroth system has a unique absolutely continuous invariant measure for which the density can be obtained from Theorem~\ref{thmm}.
\end{nrem}

\bibliographystyle{alpha}
\bibliography{random}

\end{document}